\documentclass[final,3p]{elsarticle}
 \usepackage[latin1]{inputenc}
 \usepackage{graphics}
 \usepackage{graphicx}
 \usepackage{epsfig}
\usepackage{amssymb}
 \usepackage{amsthm}
 \usepackage{lineno}
 \usepackage{amsmath}
   \numberwithin{equation}{section}
\usepackage{mathrsfs}

\newtheorem{thm}{Theorem}[section]

\newtheorem{lem}[thm]{Lemma}
\newtheorem{prop}[thm]{Proposition}

\biboptions{sort&compress,square}
\begin{document}
\begin{frontmatter}
\author[rvt1]{Jian Wang}
\ead{wangj484@nenu.edu.cn}
\author[rvt2]{Yong Wang\corref{cor2}}
\ead{wangy581@nenu.edu.cn}
\author[rvt2]{Tong Wu}
\author[rvt2]{Yuchen Yang}
\author[rvt1]{Wenying Li}
\cortext[cor2]{Corresponding author.}
\address[rvt1]{School of Science, Tianjin University of Technology and Education, Tianjin, 300222, P.R.China}
\address[rvt2]{School of Mathematics and Statistics, Northeast Normal University,
Changchun, 130024, P.R.China}

\title{One-forms, spectral Einstein functionals and the noncommutative residue  }
\begin{abstract}
For two one-forms and the Dirac operator, Dabrowski etc. recovered the spectral Einstein functionals by computing their
noncommutative residue in Theorem 4.1 \cite{DL}. In this paper, we generalize the results of Dabrowski etc. to the
cases of four dimensional spin manifolds with boundary.
\end{abstract}
\end{frontmatter}
\section{Introduction}
\label{1}

The theory of noncommutative residue  for one-dimensional manifolds was discovered by Manin \cite{YIM} and Adler \cite{MA} in connection
 with geometric aspects of nonlinear partial differential equations.
For arbitrary closed compact $n$-dimensional manifolds, the noncommutative reside was introduced by
Wodzicki in \cite{Wo,Wo1} using the theory of zeta functions of elliptic pseudodifferential operators.
Let $E$ be a finite-dimensional complex vector bundle over a closed compact manifold $M$
of dimension $n$, the noncommutative residue of a pseudo-differential operator
$P\in\Psi DO(E)$ can be defined by
 \begin{equation}
{\rm  Wres}(P):=(2\pi)^{-n}\int_{S^{*}M}\mathrm{Tr}(\sigma_{-n}^{P}(x,\xi))\mathrm{d}x \mathrm{d}\xi,
\end{equation}
where $S^{*}M\subset T^{*}M$ denotes the co-sphere bundle on $M$ and
$\sigma_{-n}^{P}$ is the component of order $-n$ of the complete symbol
 \begin{equation}
\sigma^{P}:=\sum_{i}\sigma_{i}^{P}
\end{equation}
of $P$, and the linear functional ${\rm  Wres}: \Psi DO(E)\rightarrow \mathbb{C }$
is in fact the unique trace (up to multiplication
by constants) on the algebra of pseudo-differential operators $\Psi DO(E)$.

In \cite{Co1}, Connes  computed a conformal four-dimensional
 Polyakov action analogy using the noncommutative residue.
Connes  proved that the noncommutative residue on a compact manifold $M$ coincided with Dixmier's trace on pseudodifferential
operators of order -dim$M$\cite{Co2}\cite{Co3}. The theory has very rich structures both in physics and mathematics.
More precisely, Connes made a challenging observation that
the Wodzicki residue of the inverse square of the Dirac operator yields the
Einstein-Hilbert action of general relativity. Kastler\cite{Ka} gave a brute-force proof of this theorem, and Kalau and
 Walze\cite{KW} proved
this theorem in the normal coordinates system simultaneously, which is called the Kastler-Kalau-Walze theorem now.
Let $s$ be the scalar curvature and Wres denote  the noncommutative residue, then the Kastler-Kalau-Walze
theorem gives an operator-theoretic explanation of the gravitational action and says that for a $4-$dimensional closed spin manifold and Dirac operator $D$,
 there exists a constant $c_0$, such that
 \begin{equation*}
{\rm  Wres}(D^{-2})=c_0\int_Ms{\rm dvol}_M.
\end{equation*}
On the other hand, Fedosov etc. defined a noncommutative residue on Boutet de Monvel's algebra and proved that it was a
unique continuous trace in \cite{FGLS}, and generalized the definition
of noncommutative residue to  manifolds with boundary.
 In \cite{SCH}, Schrohe gave the relation between the Dixmier trace and the noncommutative residue for
manifolds with boundary. For elliptic pseudodifferential operators, Wang proved the Kastler-Kalau-Walze type theorem and
gave the operator-theoretic explanation of the gravitational action for lower dimensional manifolds with boundary \cite{Wa1,Wa3,Wa4}.

In the noncommutative realm the spectral-theoretic approach to scalar curvature has been
extended also to quantum tori in the seminal work of Connes and Moscovici\cite{CoM}.
Furthermore, the pseudodifferential operators and symbol calculus
introduced in \cite{Co4} and extended to crossed product algebras in \cite{SBa1}\cite{SBa2}, have been employed for computations of
certain values and residues of zeta functions of suitable Laplace type operators.
Recently, in order to recover other important tensors in both the classical
setup as well as for the generalised or quantum geometries, for the metric tensor $g$, Ricci curvature $Ric$ and the scalar curvature $s$,
 Dabrowski etc. \cite{DL} defined  bilinear functionals
  \begin{equation}
G:=Ric-\frac{1}{2}s(g)g,
\end{equation}
and they demonstrated that the noncommutative residue density recovered
the tensors $g$ and $G$ as certain bilinear functionals of vector fields on a manifold $M$, while their
dual tensors are recovered as a density of bilinear functionals of differential one-forms on $M$.
Motivated by the spectral functionals over the dual bimodule of one-forms in Theorem 4.1 \cite{DL} and
the Kastler-Kalau-Walze type theorem\cite{Ka}\cite{KW},
we give some new spectral functionals which are the extension of spectral functionals for Dirac operator with Clifford multiplication
 by the local coframe basis, and we relate them to the noncommutative residue for manifolds with boundary.
  For lower dimensional compact Riemannian manifolds  with  boundary, we compute the residue
  of the type I operator ${\rm{\widetilde{Wres}}}\Big(\pi^+\big(c(w)(Dc(v)+c(v)D)D^{-1}\big)\circ \pi^+(D^{-2})\Big)$,
  the residue  of the type II operator  ${\rm{\widetilde{Wres}}}\Big(\pi^+\big(c(w)(Dc(v)+c(v)D)D^{-2}\big)\circ \pi^+(D^{-1})\Big)$
 and obtain the Dabrowski-Sitarz-Zalecki type theorems for four  dimensional spin manifolds with boundary, which we note that
 generalize the Theorem 4.1 in \cite{DL}. In \cite{JYT}, we have already generalized the Theorem 3.2 and the Lemma 3.3 in \cite{DL}
  to the cases of Dirac
 operators with torsion for four  dimensional spin manifolds with boundary.

\section{Spectral functionals over the dual bimodule of one-forms}
 \subsection{Boutet de Monvel's calculus}

To define the lower dimensional volume, some basic facts and formulae about Boutet de Monvel's calculus which can be found  in Sec.2 in \cite{Wa1}
are needed.

Let $$ F:L^2({\bf R}_t)\rightarrow L^2({\bf R}_v);~F(u)(v)=\int e^{-ivt}u(t){\rm d}t$$ denote the Fourier transformation and
$\varphi(\overline{{\bf R}^+}) =r^+\varphi({\bf R})$ (similarly define $\varphi(\overline{{\bf R}^-}$)), where $\varphi({\bf R})$
denotes the Schwartz space and
  \begin{equation}
r^{+}:C^\infty ({\bf R})\rightarrow C^\infty (\overline{{\bf R}^+});~ f\rightarrow f|\overline{{\bf R}^+};~
 \overline{{\bf R}^+}=\{x\geq0;x\in {\bf R}\}.
\end{equation}
We define $H^+=F(\varphi(\overline{{\bf R}^+}));~ H^-_0=F(\varphi(\overline{{\bf R}^-}))$ which are orthogonal to each other. We have the following
 property: $h\in H^+~(H^-_0)$ iff $h\in C^\infty({\bf R})$ which has an analytic extension to the lower (upper) complex
half-plane $\{{\rm Im}\xi<0\}~(\{{\rm Im}\xi>0\})$ such that for all nonnegative integer $l$,
 \begin{equation}
\frac{{\rm d}^{l}h}{{\rm d}\xi^l}(\xi)\sim\sum^{\infty}_{k=1}\frac{{\rm d}^l}{{\rm d}\xi^l}(\frac{c_k}{\xi^k})
\end{equation}
as $|\xi|\rightarrow +\infty,{\rm Im}\xi\leq0~({\rm Im}\xi\geq0)$.

 Let $H'$ be the space of all polynomials and $H^-=H^-_0\bigoplus H';~H=H^+\bigoplus H^-.$ Denote by $\pi^+~(\pi^-)$ respectively the
 projection on $H^+~(H^-)$. For calculations, we take $H=\widetilde H=\{$rational functions having no poles on the real axis$\}$ ($\tilde{H}$
 is a dense set in the topology of $H$). Then on $\tilde{H}$,
 \begin{equation}
\pi^+h(\xi_0)=\frac{1}{2\pi i}\lim_{u\rightarrow 0^{-}}\int_{\Gamma^+}\frac{h(\xi)}{\xi_0+iu-\xi}{\rm d}\xi,
\end{equation}
where $\Gamma^+$ is a Jordan close curve included ${\rm Im}\xi>0$ surrounding all the singularities of $h$ in the upper half-plane and
$\xi_0\in {\bf R}$. Similarly, define $\pi'$ on $\tilde{H}$,
 \begin{equation}
\pi'h=\frac{1}{2\pi}\int_{\Gamma^+}h(\xi){\rm d}\xi.
\end{equation}
So, $\pi'(H^-)=0$. For $h\in H\bigcap L^1(R)$, $\pi'h=\frac{1}{2\pi}\int_{R}h(v)\texttt{d}v$ and for $h\in H^+\bigcap L^1(R)$, $\pi'h=0$.
Denote by $\mathcal{B}$ Boutet de Monvel's algebra (for details, see Section 2 of \cite{Wa1}).

An operator of order $m\in {\bf Z}$ and type $d$ is a matrix
$$A=\left(\begin{array}{lcr}
  \pi^+P+G  & K  \\
   T  &  S
\end{array}\right):
\begin{array}{cc}
\   C^{\infty}(X,E_1)\\
 \   \bigoplus\\
 \   C^{\infty}(\partial{X},F_1)
\end{array}
\longrightarrow
\begin{array}{cc}
\   C^{\infty}(X,E_2)\\
\   \bigoplus\\
 \   C^{\infty}(\partial{X},F_2)
\end{array}.
$$
where $X$ is a manifold with boundary $\partial X$ and
$E_1,E_2~(F_1,F_2)$ are vector bundles over $X~(\partial X
)$.~Here,~$P:C^{\infty}_0(\Omega,\overline {E_1})\rightarrow
C^{\infty}(\Omega,\overline {E_2})$ is a classical
pseudodifferential operator of order $m$ on $\Omega$, where
$\Omega$ is an open neighborhood of $X$ and
$\overline{E_i}|X=E_i~(i=1,2)$. $P$ has an extension:
$~{\cal{E'}}(\Omega,\overline {E_1})\rightarrow
{\cal{D'}}(\Omega,\overline {E_2})$, where
${\cal{E'}}(\Omega,\overline {E_1})~({\cal{D'}}(\Omega,\overline
{E_2}))$ is the dual space of $C^{\infty}(\Omega,\overline
{E_1})~(C^{\infty}_0(\Omega,\overline {E_2}))$. Let
$e^+:C^{\infty}(X,{E_1})\rightarrow{\cal{E'}}(\Omega,\overline
{E_1})$ denote extension by zero from $X$ to $\Omega$ and
$r^+:{\cal{D'}}(\Omega,\overline{E_2})\rightarrow
{\cal{D'}}(\Omega, {E_2})$ denote the restriction from $\Omega$ to
$X$, then define
$$\pi^+P=r^+Pe^+:C^{\infty}(X,{E_1})\rightarrow {\cal{D'}}(\Omega,
{E_2}).$$
In addition, $P$ is supposed to have the
transmission property; this means that, for all $j,k,\alpha$, the
homogeneous component $p_j$ of order $j$ in the asymptotic
expansion of the
symbol $p$ of $P$ in local coordinates near the boundary satisfies:
$$\partial^k_{x_n}\partial^\alpha_{\xi'}p_j(x',0,0,+1)=
(-1)^{j-|\alpha|}\partial^k_{x_n}\partial^\alpha_{\xi'}p_j(x',0,0,-1),$$
then $\pi^+P:C^{\infty}(X,{E_1})\rightarrow C^{\infty}(X,{E_2})$
by Section 2.1 of \cite{Wa1}.

Now we recall the main theorem in \cite{FGLS}.
\begin{thm}\label{th:32}{\bf(Fedosov-Golse-Leichtnam-Schrohe)}
 Let $X$ and $\partial X$ be connected, ${\rm dim}X=n\geq3$,
 $A=\left(\begin{array}{lcr}\pi^+P+G &   K \\
T &  S    \end{array}\right)$ $\in \mathcal{B}$ , and denote by $p$, $b$ and $s$ the local symbols of $P,G$ and $S$ respectively.
 Define:
 \begin{align}
{\rm{\widetilde{Wres}}}(A)=&\int_X\int_{\bf S}{\mathrm{Tr}}_E\left[p_{-n}(x,\xi)\right]\sigma(\xi){\rm d}x \nonumber\\
&+2\pi\int_ {\partial X}\int_{\bf S'}\left\{{\mathrm{Tr}}_E\left[({\mathrm{Tr}}b_{-n})(x',\xi')\right]+{\mathrm{Tr}}
_F\left[s_{1-n}(x',\xi')\right]\right\}\sigma(\xi'){\rm d}x',
\end{align}
Then~~ a) ${\rm \widetilde{Wres}}([A,B])=0 $, for any
$A,B\in\mathcal{B}$;~~ b) It is a unique continuous trace on
$\mathcal{B}/\mathcal{B}^{-\infty}$.
\end{thm}

\subsection{Spectral functionals over the dual bimodule of one-forms}

In this section we consider an $n$-dimensional oriented Riemannian manifold $(M, g^{M})$ with boundary $\partial_{ M}$ equipped
with a fixed spin structure. We assume that the metric $g^{M}$ on $M$ has
the following form near the boundary
 \begin{equation}
 g^{M}=\frac{1}{h(x_{n})}g^{\partial M}+{\rm d}x _{n}^{2} ,
\end{equation}
where $g^{\partial M}$ is the metric on $\partial M$. Let $U\subset
M$ be a collar neighborhood of $\partial M$ which is diffeomorphic $\partial M\times [0,1)$. By the definition of $h(x_n)\in C^{\infty}([0,1))$
and $h(x_n)>0$, there exists $\tilde{h}\in C^{\infty}((-\varepsilon,1))$ such that $\tilde{h}|_{[0,1)}=h$ and $\tilde{h}>0$ for some
sufficiently small $\varepsilon>0$. Then there exists a metric $\hat{g}$ on $\hat{M}=M\bigcup_{\partial M}\partial M\times
(-\varepsilon,0]$ which has the form on $U\bigcup_{\partial M}\partial M\times (-\varepsilon,0 ]$
 \begin{equation}
\hat{g}=\frac{1}{\tilde{h}(x_{n})}g^{\partial M}+{\rm d}x _{n}^{2} ,
\end{equation}
such that $\hat{g}|_{M}=g$.
We fix a metric $\hat{g}$ on the $\hat{M}$ such that $\hat{g}|_{M}=g$.

Recall the definition of the Dirac operator $D$.
 Let $\nabla^L$ denote the Levi-Civita connection about $g^M$. In the local coordinates $\{x_i; 1\leq i\leq n\}$ and the
 fixed orthonormal frame $\{e_1,\cdots,e_n\}$, the connection matrix $(\omega_{s,t})$ is defined by
\begin{equation}
\nabla^L(e_1,\cdots,e_n)= (e_1,\cdots,e_n)(\omega_{s,t}).
\end{equation}
 The Dirac operator is defined by
\begin{equation}
D=\sum^n_{i=1}c(e_i)\Big[e_i
-\frac{1}{4}\sum_{s,t}\omega_{s,t}(e_i)c(e_s)c(e_t)\Big].
\end{equation}
where $c(e_i)$ denotes the Clifford action.

The following Theorem of Dabrowski etc.'s Einstein functional \cite{DL} play a key role in our proof
of the Einstein functional for manifold with boundary.
let $\bar{v}$,$\bar{w}$ with the components with
respect to local coordinates $\bar{v}_{a}$ and $\bar{w}_{a}$, respectively, be two differential forms represented in
such a way as endomorphisms (matrices) $c(\bar{v}) $ and $c(\bar{w}) $ on the spinor bundle.
We assume thus that $M$ is a
$n = 2m$ dimensional spin manifold and use the Clifford representation of one-forms as 0-
order differential operators, that is, endomorphisms of a rank $2^{m}$ spinor bundle.
 Using the  operator $c(\bar{w})(Dc(\bar{v})+c(\bar{v})D)D^{-n+1}$ acting on sections of
 a vector bundle $S(TM)$ of rank $2^{m}$,
 the spectral functionals over the dual bimodule of one-forms defined by
\begin{lem}\cite{DL}
The Einstein functional equal to
 \begin{equation}
Wres\big(c(\bar{w})(Dc(\bar{v})+c(\bar{v})D)D^{-n+1}\big)=\frac{\upsilon_{n-1}}{6}2^{m}\int_{M}(Ric^{ab}
-\frac{1}{2}s(g)g^{ab})\bar{v}_{a} \bar{w}_{b}{\rm vol}_{g},
\end{equation}
where $g^{*}(\bar{v},\bar{w})=g^{ab}v_{a}w_{b}$ and $G(\bar{v},\bar{w})= (Ric^{ab}
-\frac{1}{2}s(g)g^{ab})\bar{v}_{a} \bar{w}_{b}$ denotes the Einstein tensor evaluated on the two one-forms, where
$\bar{v}=\sum_{j=1}^{n}v_{j}dx^{j} $, $\bar{w}=\sum_{l=1}^{n}w_{l}dx^{l} $ and $\upsilon_{n-1}=\frac{2\pi^{m}}{\Gamma(m)}$ .
\end{lem}

Let $\bar{v}=\sum_{j=1}^{n}v_{j}e^{j,*} $, $\bar{w}=\sum_{l=1}^{n}w_{l}e^{l,*} $, where $\{e^{1,*},e^{2,*},\cdots,e^{n,*}  \}$
is the orthogonal basis about $g^{TM,*}$. Let $v=\sum_{j=1}^{n}\bar{v}(e_{j })e_{j }=:\sum_{j=1}^{n}v_{j}e_{j} $,
$w =:\sum_{l=1}^{n}w_{l}e_{l} $ be the vector fields dual to one forms $\bar{v},\bar{w} $.
By  definition $c(\bar{v})=c(v) $, $c(\bar{w})=c(w) $ and Lemma 2.2, we get

\begin{lem}
The Einstein functional equal to
 \begin{equation}
Wres\big(c(w)(Dc(v)+c(v)D)D^{-n+1}\big)=\frac{\upsilon_{n-1}}{6}2^{m}\int_{M}(Ric(v,w)
-\frac{1}{2}s g(v,w)) {\rm  vol}_{g},
\end{equation}
where $g(v,w)$ denotes the inner product evaluated on the two vector fields.
\end{lem}
Using an explicit formula for  the spectral functionals  of above operator, we can
reformulate the noncommutative residue for manifold $(M, g^{M})$ with boundary $\partial_{ M}$ as follows.
\begin{prop}
For the  type-I operator, the Einstein functional for four dimensional spin  manifolds with boundary defined by
\begin{align}
{\rm{\widetilde{Wres}}}\Big(\pi^+\big(c(w)(Dc(v)+c(v)D)D^{-1}\big)\circ \pi^+(D^{-2})\Big)
=Wres\big(c(w)(Dc(v)+c(v)D)D^{-3}\big)+\int_{\partial M}\Phi,
\end{align}
where
\begin{align}
\label{b15}
\Phi &=\int_{|\xi'|=1}\int^{+\infty}_{-\infty}\sum^{\infty}_{j, k=0}\sum\frac{(-i)^{|\alpha|+j+k+1}}{\alpha!(j+k+1)!}
\times {\rm Tr}_{S(TM)}[\partial^j_{x_n}\partial^\alpha_{\xi'}\partial^k_{\xi_n}\sigma^+_{r}
(c(w)(Dc(v)+c(v)D)D^{-1})(x',0,\xi',\xi_n)
\nonumber\\
&\times\partial^\alpha_{x'}\partial^{j+1}_{\xi_n}\partial^k_{x_n}\sigma_{l}(D^{-2})(x',0,\xi',\xi_n)]{\rm d}\xi_n\sigma(\xi'){\rm d}x',
\end{align}
and the sum is taken over $r+l-k-j-|\alpha|=-3$.
\end{prop}
\begin{prop}
For the  type-II  operator, the Einstein functional for four dimensional spin manifolds with boundary equal to
\begin{align}
\label{b14}
{\rm{\widetilde{Wres}}}\Big(\pi^+\big(c(w)(Dc(v)+c(v)D)D^{-2}\big)\circ \pi^+(D^{-1})\Big)
=Wres\big(c(w)(Dc(v)+c(v)D)D^{-3}\big)+\int_{\partial M}\widetilde{\Phi},
\end{align}
where
\begin{align}
\label{b15}
\widetilde{\Phi} &=\int_{|\xi'|=1}\int^{+\infty}_{-\infty}\sum^{\infty}_{j, k=0}\sum\frac{(-i)^{|\alpha|+j+k+1}}{\alpha!(j+k+1)!}
\times {\rm Tr}_{S(TM)}[\partial^j_{x_n}\partial^\alpha_{\xi'}\partial^k_{\xi_n}\sigma^+_{r}
(c(w)(Dc(v)+c(v)D)D^{-2})(x',0,\xi',\xi_n)
\nonumber\\
&\times\partial^\alpha_{x'}\partial^{j+1}_{\xi_n}\partial^k_{x_n}\sigma_{l}(D^{-1})(x',0,\xi',\xi_n)]{\rm d}\xi_n\sigma(\xi'){\rm d}x',
\end{align}
and the sum is taken over $r+l-k-j-|\alpha|=-3$.
\end{prop}
\section{Residue for type-I operator}
The aim of this section is to prove the following proposition and get the Dabrowski-Sitarz-Zalecki type theorems.
\begin{prop}
For the Laplace type-I operator, the Einstein functional for four dimensional spin  manifolds with boundary defined by
\begin{align}
{\rm{\widetilde{Wres}}}\Big(\pi^+\big(c(w)(Dc(v)+c(v)D)D^{-1}\big)\circ \pi^+(D^{-2})\Big)
=&{\rm{\widetilde{Wres}}}\Big(\pi^+\big(c(w)c(e_{j})c(\nabla^{ TM}_{e_{j}}v)D^{-1}\big)\circ \pi^+(D^{-2})\Big)\nonumber\\
&+{\rm{\widetilde{Wres}}}\Big(\pi^+\big(-2c(w)\nabla^{ S(TM)}_{v}D^{-1}\big)\circ \pi^+(D^{-2})\Big).
\end{align}
\end{prop}
\begin{proof}
Using the proposition of Clifford connections, we see that the left hand side of (3.1) equals
\begin{align}
Dc(v)&=c(v)D+[D,c(v)]\nonumber\\
=&c(v)D+\Big[\sum_{j=1}^{n}c(e_{j})\nabla^{S(TM)}_{e_{j}},c(v)\Big]\nonumber\\
                     =&c(v)D+\sum_{j=1}^{n}c(e_{j})\nabla^{S(TM)}_{e_{j}}c(v)-c(v)\sum_{j=1}^{n}c(e_{j})\nabla^{S(TM)}_{e_{j}}\nonumber\\
                     =&c(v)D+\sum_{j=1}^{n}c(e_{j})\Big[\nabla^{S(TM)}_{e_{j}},c(v)\Big]+\sum_{j=1}^{n}c(e_{j})c(v)\nabla^{ S(TM)}_{e_{j}}
                           -c(v)\sum_{j=1}^{n}c(e_{j})\nabla^{S(TM)}_{e_{j}}\nonumber\\
                      =&c(v)D+\sum_{j=1}^{n}c(e_{j})\Big[\nabla^{S(TM)}_{e_{j}},c(v)\Big]
                      -c(v)\sum_{j=1}^{n}c(e_{j})\nabla^{S(TM)}_{e_{j}}\nonumber\\
                                            &+\sum_{j=1}^{n}\big(-c(v)c(e_{j})-2g(e_{j},v)\big)\nabla^{S(TM)}_{e_{j}}\nonumber\\
                        =&c(v)D+\sum_{j=1}^{n}c(e_{j})c(\nabla^{TM}_{e_{j}}v)-2c(v)D-2\sum_{j=1}^{n}g(e_{j},v)\nabla^{ S(TM)}_{e_{j}}\nonumber\\
                     =&\sum_{j=1}^{n}c(e_{j})c(\nabla^{TM}_{e_{j}}v)-c(v)D-2\nabla^{ S(TM)}_{v}.
\end{align}
An easy calculation gives
\begin{align}
c(w)\big(Dc(v)+c(v)D\big)D^{-1}&=c(w)Dc(v)D^{-1}+c(w)c(v)DD^{-1}\nonumber\\
  &=c(w)Dc(v)D^{-1}+c(w)c(v)\nonumber\\
  &=c(w)\Big(c(v)D+\sum_{j=1}^{n}c(e_{j})c(\nabla^{TM}_{e_{j}}v)-c(v)D-2\nabla^{ S(TM)}_{v}\Big)D^{-1}\nonumber\\
               &=\sum_{j=1}^{n}c(w)c(e_{j})c(\nabla^{TM}_{e_{j}}v)D^{-1}-2c(w)\nabla^{ S(TM)}_{v}D^{-1},
\end{align}
 and the proof of the Proposition is complete.
\end{proof}
Combining with the generating Proposition 3.1, this yields
\begin{align}
&{\rm{\widetilde{Wres}}}\Big(\pi^+\big(\sum_{j=1}^{n}c(w)c(e_{j})c(\nabla^{ TM}_{e_{j}}v)D^{-1}\big)\circ \pi^+(D^{-2})\Big)\nonumber\\
=&Wres\Big( \sum_{j=1}^{n}c(w)c(e_{j})c(\nabla^{ TM}_{e_{j}}v)D^{-3} \Big)+\int_{\partial M}\Phi^{*},
\end{align}
and
\begin{align}
 {\rm{\widetilde{Wres}}}\Big(\pi^+\big(-2c(w)\nabla^{ S(TM)}_{v}D^{-1}\big)\circ \pi^+(D^{-2})\Big)
= Wres\Big( -2c(w)\nabla^{ S(TM)}_{v}D^{-3} \Big)+\int_{\partial M}\widetilde{\Phi^{*}}.
\end{align}

Since $\Phi^{*},\widetilde{\Phi^{*}}$ is a global form on $\partial M$, so for any
fixed point $x_0\in\partial M$, we can choose the normal coordinates
$U$ of $x_0$ in $\partial M$ (not in $M$) and compute $\Phi^{*}(x_0),\widetilde{\Phi^{*}}(x_0)$ in
the coordinates $\widetilde{U}=U\times [0,1)\subset M$ and the
metric $\frac{1}{h(x_n)}g^{\partial M}+dx_n^2.$ The dual metric of
$g^M$ on $\widetilde{U}$ is ${h(x_n)}g^{\partial M}+dx_n^2.$  Write
$g^M_{ij}=g^M(\frac{\partial}{\partial x_i},\frac{\partial}{\partial
x_j});~ g_M^{ij}=g^M(dx_i,dx_j)$, then
$$[g^M_{i,j}]= \left[\begin{array}{lcr}
  \frac{1}{h(x_n)}[g_{i,j}^{\partial M}]  & 0  \\
   0  &  1
\end{array}\right];~~~
[g_M^{i,j}]= \left[\begin{array}{lcr}
  h(x_n)[g^{i,j}_{\partial M}]  & 0  \\
   0  &  1
\end{array}\right], $$
and
$$\partial_{x_s}g_{ij}^{\partial M}(x_0)=0, 1\leq i,j\leq
n-1; ~~~g_{ij}^M(x_0)=\delta_{ij}. $$

 Let $n=4$
and $\{e_1,\cdots,e_{n-1}\}$ be an orthonormal frame field in $U$
about $g^{\partial M}$ which is parallel along geodesics and
$e_i(x_0)=\frac{\partial}{\partial x_i}(x_0)$, then
$\{\widetilde{e_1}=\sqrt{h(x_n)}e_1,\cdots,\widetilde{e_{n-1}}=\sqrt{h(x_n)}e_{n-1},
\widetilde{e_n}=dx_n\}$ is the orthonormal frame field in
$\widetilde U$ about $g^M$. Locally $S(TM)|_{\widetilde {U}}\cong
\widetilde {U}\times\wedge^* _{\bf C}(\frac{n}{2}).$ Let
$\{f_1,\cdots,f_4\}$ be the orthonormal basis of $\wedge^* _{\bf
C}(\frac{n}{2}).$ Take a spin frame field $\sigma:~\widetilde
{U}\rightarrow {\rm Spin}(M)$ such that $\pi\sigma=
\{\widetilde{e_1},\cdots,\widetilde{e_n}\}$ where $\pi :~{\rm
Spin}(M)\rightarrow O(M)$ is a double covering, then
$\{[(\sigma,f_i)],~1\leq i\leq 4\}$ is an orthonormal frame of
$S(TM)|_{\widetilde {U}}.$ In the following, since the global form
$\Phi$ is independent of the choice of the local frame, so we can
compute ${\rm tr}_{S(TM)}$ in the frame $\{[(\sigma,f_i)],~1\leq
i\leq 4\}.$ Let $\{E_1,\cdots,E_n\}$ be the canonical basis of ${\bf
R}^n$ and $c(E_i)\in {\rm cl}_{\bf C}(n)\cong {\rm Hom}(\wedge^*
_{\bf C}(\frac{n}{2}),\wedge^* _{\bf C}(\frac{n}{2}))$ be the
Clifford action, then
$$c( \widetilde{e_i})=[(\sigma,c(E_i))];~ c( \widetilde{e_i})[(\sigma,f_i)]=[(\sigma,c(E_i)f_i)];~
\frac{\partial}{\partial x_i}=[(\sigma,\frac{\partial}{\partial
x_i})], $$
then we have $\frac{\partial}{\partial x_i}c( \widetilde{e_i})=0$ in the above frame.

\subsection{The computation of $\Phi^{*}$}
\begin{lem}\cite{Wa3,Wa4} The following identities hold:
\begin{align}
&\sigma_{-1}(D^{-1})=\frac{\sqrt{-1}c(\xi)}{|\xi|^2};\\
&\sigma_{-2}(D^{-2})=|\xi|^{-2}.
\end{align}
\end{lem}
\begin{lem} The following identities hold:
\begin{align}
&{\rm Tr} \Big(\sum_{j=1}^{n}c(w)c(e_{j})c(\nabla^{ TM}_{e_{j}}v)c(dx_{n})\Big)(x_{0})\nonumber\\
=&\Big(\sum_{j=1}^{n}g(e_{j},\nabla^{ TM}_{e_{j}}v)g(w,e_{n})-g(w,\nabla^{ TM}_{e_{n}}v)
+g(\nabla^{ TM}_{w}v, e_{n}) \Big){\rm Tr}_{S(TM)}[{\rm id}].
\end{align}
\end{lem}
\begin{proof}
By using inner product formulas, we obtain
\begin{align}
&{\rm Tr} \Big(\sum_{j=1}^{n}c(w)c(e_{j})c(\nabla^{ TM}_{e_{j}}v)c(dx_{n})\Big)(x_{0})\nonumber\\
=&\sum_{j=1}^{n}{\rm Tr} \Big(-c(e_{j})c(w)c(\nabla^{ TM}_{e_{j}}v)c(dx_{n})\Big)(x_{0})
-2\sum_{j=1}^{n}g(e_{j},w){\rm Tr}\Big(c(\nabla^{ TM}_{e_{j}}v)c(dx_{n})\Big)\nonumber\\
=&\sum_{j=1}^{n}{\rm Tr} \Big(c(e_{j})c(\nabla^{ TM}_{e_{j}}v)c(w)c(dx_{n})\Big)(x_{0})
+2\sum_{j=1}^{n}g(w,\nabla^{ TM}_{e_{j}}v){\rm Tr}\Big(c(e_{j})c(dx_{n})\Big)\nonumber\\
&-2\sum_{j=1}^{n}g(e_{j},w){\rm Tr}\Big(c(\nabla^{ TM}_{e_{j}}v)c(dx_{n})\Big)\nonumber\\
=&\sum_{j=1}^{n}{\rm Tr} \Big(-c(e_{j})c(\nabla^{ TM}_{e_{j}}v)c(dx_{n})c(w)\Big)(x_{0})
-2\sum_{j=1}^{n}g(w,dx_{n}){\rm Tr}\Big(c(e_{j})\nabla^{ TM}_{e_{j}}v\Big)\nonumber\\
&+2\sum_{j=1}^{n}g(w,\nabla^{ TM}_{e_{j}}v){\rm Tr}\Big(c(e_{j})c(dx_{n})\Big)(x_{0})
-2\sum_{j=1}^{n}g(e_{j},w){\rm Tr}\Big(c(\nabla^{ TM}_{e_{j}}v)c(dx_{n})\Big).
\end{align}
Then making the term of shift, and the proof of the lemma is complete.
\end{proof}

Now we can compute $\Phi^{*}$, since the sum is taken over $
-r-l+k+j+|\alpha|=3,~~r\leq-1,l\leq-2$, we get $r=-1,l=-2,~k=|\alpha|=j=0,$ then

\begin{align}
\Phi^{*}=-i\int_{|\xi'|=1}\int^{+\infty}_{-\infty}{\rm Tr} \Big[\pi^+_{\xi_n}\sigma_{-1}
\Big(\sum_{j=1}^{n}c(w)c(e_{j})c(\nabla^{ TM}_{e_{j}}v)D^{-1}\Big)\times
\partial_{\xi_n}\sigma_{-2}(D^{-2})\Big](x_0){\rm d}\xi_n\sigma(\xi'){\rm d}x'.
\end{align}
An easy calculation gives
\begin{align}
\pi^+_{\xi_n}\sigma_{-1}\Big(\sum_{j=1}^{n}c(w)c(e_{j})c(\nabla^{ TM}_{e_{j}}v)D^{-1}\Big)
=\sum_{j=1}^{n}c(w)c(e_{j})c(\nabla^{ TM}_{e_{j}}v)\frac{c(\xi')+ic(dx_n)}{2(\xi_n-i)},
\end{align}
and
\begin{align}
\partial_{\xi_n}\sigma_{-2}(D^{-2})=\frac{-2\xi_n}{(1+\xi_n^2)^2}.
\end{align}
It follows that
\begin{align}
&{\rm Tr} \Big[\pi^+_{\xi_n}\sigma_{-1}\Big(\sum_{j=1}^{n}c(w)c(e_{j})c(\nabla^{ TM}_{e_{j}}v)D^{-1}\Big)\times
\partial_{\xi_n}\sigma_{-2}(D^{-2})\Big](x_0)\nonumber\\
=& \frac{-\xi_n}{(\xi_n-i)^3(\xi_n+i)^2}{\rm Tr} \Big(\sum_{j=1}^{n}c(w)c(e_{j})c(\nabla^{ TM}_{e_{j}}v)c(\xi')\Big)(x_0)\nonumber\\
&-\frac{i\xi_n}{(\xi_n-i)^3(\xi_n+i)^2} {\rm Tr} \Big(\sum_{j=1}^{n}c(w)c(e_{j})c(\nabla^{ TM}_{e_{j}}v)c(dx_{n})\Big)(x_0).
\end{align}
From (3.12),(3.13) and Lemma 3.3, we obtain
\begin{align}
\Phi^{*}=&-i\int_{|\xi'|=1}\int^{+\infty}_{-\infty} \Big[-\frac{i\xi_n}{(\xi_n-i)^3(\xi_n+i)^2} {\rm Tr}
 \Big(\sum_{j=1}^{n}c(w)c(e_{j})c(\nabla^{ TM}_{e_{j}}v)c(dx_{n})\Big)\Big](x_0){\rm d}\xi_n\sigma(\xi'){\rm d}x'\nonumber\\
  =&\int_{\Gamma^{+}}-\frac{i\xi_n}{(\xi_n-i)^3(\xi_n+i)^2}{\rm Tr}
 \Big(\sum_{j=1}^{n}c(w)c(e_{j})c(\nabla^{ TM}_{e_{j}}v)c(dx_{n})\Big)\Big](x_0)\Omega_3 {\rm d}x'\nonumber\\
 =&-2\pi^2 \Big(\sum_{j=1}^{n}g(e_{j},\nabla^{ TM}_{e_{j}}v)g(w,e_{n})-g(w,\nabla^{ TM}_{e_{n}}v)+g(\nabla^{ TM}_{w}v, e_{n}) \Big){\rm d}x'.
\end{align}

\subsection{The computation of $\widetilde{\Phi^{*}}$}
Let $v=\sum_{j=1}^nv_j\partial_{x_j}$,
denote the  spin connection
by  $\nabla_v^{S(TM)}:=v+\frac{1}{4}\Sigma_{ij}\langle\nabla_v^L{\tilde{e}_i},\tilde{e}_j\rangle c(\tilde{e}_i)c(\tilde{e}_j)=v+ A(v)$.
By $\sigma(\partial_{x_j})=-\sqrt{-1}\xi_j$, we have the following lemmas.
\begin{lem}\label{lem3} The following identities hold:
\begin{align}
\label{b22}
\sigma_{0}(\nabla_v^{S(TM)}D^{-1})=&\sum_{j=1}^nv_j\sqrt{-1}\xi_{j}\frac{\sqrt{-1}c(\xi)}{|\xi|^2};\\
\sigma_{-1}(\nabla_v^{S(TM)}D^{-1})=&\sigma_{1}(\nabla_v^{S(TM)} )\sigma_{-2}(D^{-1})
+\sigma_{0}(\nabla_v^{S(TM)} )\sigma_{-1}(D^{-1})\nonumber\\
&+\sum_{j=1}^{n}\partial_{\xi_{j}}\sigma_{1}\big(\nabla_v^{S(TM)} \big)
D_{x_{j}}\big(\sigma_{-1}(D^{-1})\big).
\end{align}
\end{lem}
\begin{proof}
Write
 \begin{eqnarray}
D_x^{\alpha}&=(-i)^{|\alpha|}\partial_x^{\alpha};
~\sigma(D)=p_1+p_0;
~\sigma(D)^{-1}=\sum^{\infty}_{j=1}q_{-j}.
\end{eqnarray}
 By the composition formula of pseudodifferential operators, we have
\begin{align}
1=\sigma(D\circ D^{-1})&=\sum_{\alpha}\frac{1}{\alpha!}\partial^{\alpha}_{\xi}[\sigma(D)]
D_x^{\alpha}[\sigma(D^{-1})]\nonumber\\
&=(p_1+p_0)(q_{-1}+q_{-2}+q_{-3}+\cdots)\nonumber\\
&~~~+\sum_j(\partial_{\xi_j}p_1+\partial_{\xi_j}p_0)(
D_{x_j}q_{-1}+D_{x_j}q_{-2}+D_{x_j}q_{-3}+\cdots)\nonumber\\
&=p_1q_{-1}+(p_1q_{-2}+p_0q_{-1}+\sum_j\partial_{\xi_j}p_1D_{x_j}q_{-1})+\cdots,
\end{align}
so
\begin{equation}
q_{-1}=p_1^{-1};~q_{-2}=-p_1^{-1}[p_0p_1^{-1}+\sum_j\partial_{\xi_j}p_1D_{x_j}(p_1^{-1})].
\end{equation}
Then,
\begin{align}
\label{b22}
\sigma_{-1}(\nabla_v^{S(TM)}D^{-1})=&\sigma_{1}(\nabla_v^{S(TM)} )\sigma_{-2}(D^{-1})
+\sigma_{0}(\nabla_v^{S(TM)} )\sigma_{-1}(D^{-1})\nonumber\\
&+\sum_{j=1}^{n}\partial_{\xi_{j}}\sigma_{1}\big(\nabla_v^{S(TM)} \big)
D_{x_{j}}\big(\sigma_{-1}(D^{-1})\big).
\end{align}
\end{proof}

\begin{lem}\cite{Wa3,Wa4}The symbol of the Dirac operator
\begin{align}
&\sigma_{-2}(D^{-1})=\frac{c(\xi)\sigma_{0}(D)c(\xi)}{|\xi|^4}+\frac{c(\xi)}{|\xi|^6}\sum_jc(dx_j)
[\partial_{x_j}[c(\xi)]|\xi|^2-c(\xi)\partial_{x_j}(|\xi|^2)];\nonumber\\
&\sigma_{-3}(D^{-2})=-\sqrt{-1}|\xi|^{-4}\xi_k(\Gamma^k-2\delta^k)-\sqrt{-1}|\xi|^{-6}2\xi^j\xi_\alpha\xi_\beta\partial_jg^{\alpha\beta},
\end{align}
where  $\sigma_{0}(D)(x_0)=-\frac{3}{4}h'(0)c(dx_n)$.
\end{lem}

\begin{lem}\label{le:32}
With the metric $g^{M}$ on $M$ near the boundary
\begin{eqnarray}
\partial_{x_j}(|\xi|_{g^M}^2)(x_0)&=&\left\{
       \begin{array}{c}
        0,  ~~~~~~~~~~ ~~~~~~~~~~ ~~~~~~~~~~~~~{\rm if }~j<n; \\[2pt]
       h'(0)|\xi'|^{2}_{g^{\partial M}},~~~~~~~~~~~~~~~~~~~~~{\rm if }~j=n.
       \end{array}
    \right. \\
\partial_{x_j}[c(\xi)](x_0)&=&\left\{
       \begin{array}{c}
      0,  ~~~~~~~~~~ ~~~~~~~~~~ ~~~~~~~~~~~~~{\rm if }~j<n;\\[2pt]
\frac{h'(0)}{2} c(\xi')(x_{0}), ~~~~~~~~~~~~~~~~~~~{\rm if }~j=n,
       \end{array}
    \right.
\end{eqnarray}
where $\xi=\xi'+\xi_{n}\texttt{d}x_{n}$.
\end{lem}
\begin{proof}
By the equality
  $\partial_{x_j}(|\xi|_{g^M}^2)(x_0)=\partial_{x_j}(h(x_n)g^{l,m}_{\partial M}(x')\xi_l\xi_m+\xi_n^2)(x_0)$ and
Lemma 2.2 \cite{Wa3}, the first  is correct.

 Write $<dx^k,e_l^{*}>_{g^{\partial M}}=H^{k,l}, 1\leq k,l \leq n$, then
$\partial_{x_j}H^{k,l}(x_0)=0$. Define $dx_j^*\in TM|_{\widetilde
U}$ by $<dx_j^*,v>=(dx_j,v)$ for $v\in TM.$ For $j<n$,
\begin{align}
c(\xi')=\sum_{k,l=1}^{n-1}\xi_{k}g(dx^{k},\widetilde{e_l}^{*})c(\widetilde{e_l})
       =\sum_{k,l=1}^{n-1}\xi_{k} \sqrt{h(x_{n})}H_{k,l}c(\widetilde{e_l}).
\end{align}
Then for $l<n$, $\partial_{x_l}c(\xi')(x_0)=0, $
and
\begin{align}
\partial_{x_n}c(\xi')(x_0)=\frac{h'(x_{n})}{2\sqrt{h(x_{n})}}\sum_{k,l=1}^{n-1}\xi_{k}\delta_{k}^{l}c(\widetilde{e_l})(x_0)
  =\frac{h'(0)}{2} c(\xi')(x_{0}).
\end{align}
\end{proof}

\begin{lem} The following identities hold:
\begin{align}
\int_{|\xi'|=1}{\rm Tr} \Big(\sum_{j=1}^{n-1}\xi_{j}\partial_{x_{n}}\big( v_{j}c(w)c(\xi')\big)   \Big)(x_0)\sigma(\xi')
= -\frac{4\pi}{3}\Big(h'(0)g^{M}(v^{T},w^{T})+\partial_{x_{n}}g^{M}(v^{T},w^{T})\Big){\rm Tr}_{S(TM)}[{\rm id}](x_0)
\end{align}
\end{lem}
\begin{proof}
Let $c(\xi')=\sum_{k,l=1}^{n-1}\xi_{k}\sqrt{h(x_{n})}H^{kl}c(\tilde{e}_{l})$ and $H^{kl}=g^{\partial_{M}}(dx^{k},e^{l,*})$,
\begin{align}
&\int_{|\xi'|=1}{\rm Tr} \Big(\sum_{j=1}^{n-1}\xi_{j}\partial_{x_{n}}\big( v_{j}c(w)c(\xi')\big)   \Big)(x_{0})\sigma(\xi')\nonumber\\
=& \int_{|\xi'|=1} \sum_{j=1}^{n-1}\xi_{j}\partial_{x_{n}}{\rm Tr} \big( v_{j}c(w)c(\xi')\big) (x_{0})\sigma(\xi')\nonumber\\
=& \int_{|\xi'|=1}\sum_{j=1}^{n-1}\xi_{j}\partial_{x_{n}}{\rm Tr} \Big(  v_{j}c(w)\sum_{k,l=1}^{n-1}\xi_{k}\sqrt{h(x_{n})}H^{kl} c(\tilde{e}_{l})
\Big)(x_{0})\sigma(\xi')\nonumber\\
=& \int_{|\xi'|=1}\sum_{j,k,l=1}^{n-1}\xi_{j}\xi_{k} \Big(\partial_{x_{n}}[\sqrt{h(x_{n})}]x_{0}H^{kl}{\rm Tr}\big( v_{j}c(w)c(\tilde{e}_{l})  \big)
\Big)(x_{0})\sigma(\xi')\nonumber\\
&+\int_{|\xi'|=1}\sum_{j,k,l=1}^{n-1}\xi_{j}\xi_{k} \Big(\sqrt{h(x_{0})}H^{kl}
\partial_{x_{n}}{\rm Tr}\big( v_{j}c(w)c(\tilde{e}_{l})\big)\Big)(x_{0})\sigma(\xi')\nonumber\\
=& \int_{|\xi'|=1}\sum_{j,k,l=1}^{n-1}\xi_{j}\xi_{k} \Big( \frac{h'(0)}{2}\delta^{kl} {\rm Tr}\big( v_{j}c(w)c(\tilde{e}_{l})  \big)\Big)
(x_{0})\sigma(\xi')\nonumber\\
&+\int_{|\xi'|=1}\sum_{j,k,l=1}^{n-1}\xi_{j}\xi_{k} \Big(\delta^{kl}\partial_{x_{n}}{\rm Tr}\big( v_{j}c(w)c(\tilde{e}_{l})\big)\Big)
(x_{0})\sigma(\xi'),\nonumber\\
=& -\frac{2\pi}{3} h'(0)g^{M}(v^{T},w^{T}){\rm Tr}_{S(TM)}[{\rm id}](x_0)\nonumber\\
&+\int_{|\xi'|=1}\sum_{j,k,l=1}^{n-1}\xi_{j}\xi_{k} \Big(\delta^{kl}\partial_{x_{n}}{\rm Tr}\big( v_{j}c(w)c(\tilde{e}_{l})\big)\Big)
(x_{0})\sigma(\xi'),
     \end{align}
where $w =\sum_{k=1}^{n}w_{k}\tilde{e}_{k}$,
\begin{align}
 \sum_{j,k,l=1}^{n-1}\xi_{j}\xi_{k} \Big( \frac{h'(0)}{2}\delta^{kl} {\rm Tr}\big( v_{j}c(w)c(\tilde{e}_{l})  \big)(x_{0})\Big)
 = \sum_{j,k =1}^{n-1}\xi_{j}\xi_{k}\frac{h'(0)}{2}(-v_{j}w_{k}){\rm Tr}_{S(TM)}[{\rm id}].
     \end{align}

Let $v=\sum_{j=1}^{n}v_{j}\partial_{x_{j}}$, then
\begin{align}
v^{T}=:\sum_{j=1}^{n-1}v_{j}\partial_{x_{j}} =\sum_{j,\alpha=1}^{n-1}v_{j} \langle\partial_{x_{j}},e_\alpha\rangle _{g^{\partial M}} e_\alpha
=\sum_{j,\alpha=1}^{n-1}v_{j} H_{j\alpha}^{\partial M}e_\alpha=\sum_{j,\alpha=1}^{n}\Big(v_{j} H_{j\alpha}^{\partial M}\Big)e_\alpha
=:\sum_{\alpha=1}^{n-1}\tilde{v}_{\alpha}e_\alpha,
     \end{align}
 and  we note that
\begin{align}
\sum_{\alpha=1}^{n}\partial_{x_{n}}( \tilde{v}_{\alpha}w_{\alpha})
=\sum_{\alpha=1}^{n}\partial_{x_{n}}(\sum_{j=1}^{n}\tilde{v}_{j} H_{j\alpha}^{\partial_{M}}w_{\alpha})
=\sum_{\alpha=1}^{n}\partial_{x_{n}}( v_{\alpha}w_{\alpha}).
\end{align}
Let $w^{T}=\sum_{k=1}^{n-1}w_{k}\tilde{e}_{k}=\sum_{k=1}^{n-1}w_{k}\sqrt{h(x_{n})}e_{k}$, then
\begin{align}
 g^{M}(v^{T},w^{T})= g^{M}(\sum_{\alpha=1}^{n-1}\tilde{v}_{\alpha}e_{\alpha},\sum_{k=1}^{n-1}w_{k}\tilde{e}_{k})
 =g^{M}\Big(\sum_{k=1}^{n-1}\tilde{v}_{\alpha}\frac{\tilde{e}_{\alpha}}{\sqrt{h(x_{n})}},\sum_{k=1}^{n-1}w_{k}\tilde{e}_{k}\Big)
 =\frac{1}{\sqrt{h(x_{n})}}\sum_{j=1}^{n-1}\tilde{v}_{j}w_{j}.
 \end{align}
Therefore
\begin{align}
\int_{|\xi'|=1}\sum_{j,k,l=1}^{n-1}\xi_{j}\xi_{k}  \delta^{kl}\partial_{x_{n}}{\rm Tr}\big( v_{j}c(w)c(\tilde{e}_{l})\big)(x_0)\sigma(\xi')
=&\int_{|\xi'|=1}\sum_{j,k =1}^{n-1}\xi_{j}\xi_{k} \partial_{x_{n}} \big(- v_{j}w_{k}\big)  {\rm Tr}_{S(TM)}[{\rm id}](x_0)\sigma(\xi')\nonumber\\
=&\int_{|\xi'|=1}\sum_{j,k =1}^{n-1}\xi_{j}\xi_{k} \partial_{x_{n}} \big(- \tilde{v}_{j}w_{k}\big)  {\rm Tr}_{S(TM)}[{\rm id}](x_0)\sigma(\xi')\nonumber\\
=& -\frac{4}{3}\pi\partial_{x_{n}} \big(\sqrt{h(x_{n})} g^{M}(v^{T},w^{T})\big)  {\rm Tr}_{S(TM)}[{\rm id}](x_0)\nonumber\\
=&-\frac{2}{3}\pi h'(0)g^{M}(v^{T},w^{T}){\rm Tr}_{S(TM)}[{\rm id}](x_0)\nonumber\\
&-\frac{4}{3}\pi\partial_{x_{n}}\big(g^{M}(v^{T},w^{T}) \big){\rm Tr}_{S(TM)}[{\rm id}](x_0),
     \end{align}
      and the proof of the Lemma is complete.
\end{proof}
Now we  need to compute $  \widetilde{\Phi^{*}}$. When $n=4$, then ${\rm Tr}_{S(TM)}[{\rm \texttt{id}}]
={\rm dim}(\wedge^*(\mathbb{R}^2))=4$,
 the sum is taken over $r+l-k-j-|\alpha|=-3,~~r\leq 0,~~l\leq-2,$ then we have the following five cases:

 {\bf case a)~I)}~$r=0,~l=-2,~k=j=0,~|\alpha|=1$.

 By (2.13), we get
\begin{equation}
\label{b24}
\widetilde{\Phi^{*}}_1=-\int_{|\xi'|=1}\int^{+\infty}_{-\infty}\sum_{|\alpha|=1}
 {\rm Tr}[\partial^\alpha_{\xi'}\pi^+_{\xi_n}\sigma_{0}(-2c(w)\nabla_v^{S(TM)} D^{-1})\times
 \partial^\alpha_{x'}\partial_{\xi_n}\sigma_{-2}(D^{-2})](x_0){\rm d}\xi_n\sigma(\xi'){\rm d}x'.
\end{equation}

 By Lemma 2.2 in \cite{Wa3}, for $i<n$, then
\begin{equation}
\partial_{x_i}\sigma_{-2}({D}^{-2})(x_0)=\partial_{x_i}(|\xi|^{-2})(x_0)=-\frac{\partial_{x_i}(|\xi|^{2})(x_0)}{|\xi|^4}=0,
\end{equation}
so $\widetilde{\Phi^{*}}_1=0$.

 {\bf case a)~II)}~$r=0,~l=-2,~k=|\alpha|=0,~j=1$.

By (2.13), we get
\begin{equation}
\label{b26}
\widetilde{\Phi^{*}}_2=-\frac{1}{2}\int_{|\xi'|=1}\int^{+\infty}_{-\infty} {\rm
Tr} [\partial_{x_n}\pi^+_{\xi_n}\sigma_{0}(-2c(w)\nabla_v^{S(TM)} D^{-1})\times
\partial_{\xi_n}^2\sigma_{-2}(D^{-2})](x_0){\rm d}\xi_n\sigma(\xi'){\rm d}x'.
\end{equation}

  By Lemma 3.2, we have
\begin{eqnarray}
\partial_{\xi_n}^2\sigma_{-2}((D^{-2}))(x_0)=\partial_{\xi_n}^2(|\xi|^{-2})(x_0)=\frac{6\xi_n^2-2}{(1+\xi_n^2)^3}.
\end{eqnarray}
   By Lemma 3.4, we have
 \begin{align}
\label{b22}
\sigma_{0}(-2c(w)\nabla_v^{S(TM)}D^{-1})=&\sum_{j=1}^n(-2c(w))v_j\sqrt{-1}\xi_{j}\frac{\sqrt{-1}c(\xi)}{|\xi|^2}.
\end{align}
Also, straightforward computations yield
 \begin{align}
 \partial_{x_n}\sigma_{0}(-2c(w)\nabla_v^{S(TM)}D^{-1})
=&\partial_{x_n}\Big(\sum_{j=1}^n(-2c(w))v_j\sqrt{-1}\xi_{j}\frac{\sqrt{-1}c(\xi)}{|\xi|^2}\Big)\nonumber\\
=&2\sum_{j=1}^n\xi_{j}\partial_{x_n}\Big(c(w)v_j\xi_{j}\frac{c(\xi)}{|\xi|^2}\Big)\nonumber\\
=&2\sum_{j=1}^n\xi_{j}\Big[\partial_{x_n}\Big(c(w)v_j\frac{c(\xi')}{|\xi|^2}\Big)
  +\xi_{n}\partial_{x_n}\Big(c(w)v_j c(dx_{n})\frac{1}{|\xi|^2} \Big)\Big].
  \end{align}
Denote by ${\rm Tr} $ the trace, by the relation of the Clifford action and ${\rm Tr} (AB)={\rm Tr} (BA)$, then
we have the equalities
  \begin{align}
  {\rm Tr}\big(c(w)\partial_{x_n}(c(\xi'))\big)(x_{0})
={\rm Tr}\big(c(w)\frac{h'(0)}{2}(c(\xi'))(x_{0})
    =-\frac{h'(0)}{2}g(w,\xi')(x_{0}){\rm Tr}_{S(TM)}[{\rm  id }].
  \end{align}
In the same way we get
\begin{align}
  {\rm Tr}\Big(\partial_{x_n}\big(v_{n}c(w)c(dx_{n})\big)\Big)(x_{0})
  =&\partial_{x_n}\Big({\rm Tr}\big(v_{n}c(w)c(dx_{n})\big)\Big)(x_{0})\nonumber\\
  =&\partial_{x_n}\Big(v_{n}{\rm Tr}\big(c(w)c(dx_{n})\big)\Big)(x_{0})=\partial_{x_n}\big(-v_{n} w_{n} \big) (x_{0})
  {\rm Tr}_{S(TM)}[{\rm  id }].
  \end{align}
 We note that $i<n,~\int_{|\xi'|=1}\xi_{i_{1}}\xi_{i_{2}}\cdots\xi_{i_{2d+1}}\sigma(\xi')=0$,
so we omit some items that have no contribution for computing {\bf case a)~II)}. By Lemma 3.7 and (3.36)-(3.40), we obtain
  \begin{align}
\label{b26}
\widetilde{\Phi^{*}}_2=&-\frac{1}{2}\int_{|\xi'|=1}\int^{+\infty}_{-\infty} {\rm
Tr} [\partial_{x_n}\pi^+_{\xi_n}\sigma_{0}(-2c(w)\nabla_v^{S(TM)} D^{-1})\times
\partial_{\xi_n}^2\sigma_{-2}(D^{-2})](x_0){\rm d}\xi_n\sigma(\xi'){\rm d}x' \nonumber\\
=&-\frac{1}{2}\int_{|\xi'|=1}\int^{+\infty}_{-\infty} \frac{i(2-6\xi_{n}^{2})}{(\xi_{n}-i)(1+\xi_{n}^{2})^{3}}
\sum_{j=1}^{n-1}\xi_{j}{\rm Tr}\Big( \partial_{x_n}\big(v_{j}c(w)c(\xi')\big) \Big)(x_0){\rm d}\xi_n\sigma(\xi'){\rm d}x' \nonumber\\
&-\frac{1}{2}\int_{|\xi'|=1}\int^{+\infty}_{-\infty} \frac{(2+i\xi_{n})(3\xi_{n}^{2}-1)}{(\xi_{n}-i)^{2}(1+\xi_{n}^{2})^{3}}
\sum_{j=1}^{n-1}\xi_{j}h'(0){\rm Tr}\big( v_{j}c(w)c(\xi')\big) (x_0){\rm d}\xi_n\sigma(\xi'){\rm d}x' \nonumber\\
&-\frac{1}{2}\int_{|\xi'|=1}\int^{+\infty}_{-\infty} \frac{i(6\xi_{n}^{2}-2)}{(\xi_{n}-i)(1+\xi_{n}^{2})^{3}}
{\rm Tr}\Big( \partial_{x_n}\big(v_{n}c(w)c(dx_{n})\big) \Big)(x_0){\rm d}\xi_n\sigma(\xi'){\rm d}x' \nonumber\\
&-\frac{1}{2}\int_{|\xi'|=1}\int^{+\infty}_{-\infty} \frac{i\xi_{n}(3\xi_{n}^{2}-1)}{(\xi_{n}-i)^{2}(1+\xi_{n}^{2})^{3}}
h'(0){\rm Tr}\Big(  v_{n}c(w)c(dx_{n} \big) \Big)(x_0){\rm d}\xi_n\sigma(\xi'){\rm d}x' \nonumber\\
=&\frac{2}{3}\pi^{2} \partial_{x_n}\big(g^{M}(v^{T},w^{T})\big){\rm d}x'
  +\frac{1}{6}\pi^{2}g^{M}(v^{T},w^{T})h'(0){\rm d}x'\nonumber\\
  &+ 2\pi^{2}\partial_{x_n}\big(v_{n}w_{n}\big)dx'-\frac{1}{2}\pi^{2}v_{n}w_{n}h'(0){\rm d}x'.
 \end{align}

 {\bf case a)~III)}~$r=0,~l=-2,~j=|\alpha|=0,~k=1$.

 By (2.13), we get
\begin{align}\label{36}
\widetilde{\Phi^{*}}_3&=-\frac{1}{2}\int_{|\xi'|=1}\int^{+\infty}_{-\infty}
{\rm Tr} [\partial_{\xi_n}\pi^+_{\xi_n}
\sigma_{0}(-2c(w)\nabla_v^{S(TM)}D^{-1})
\times
\partial_{\xi_n}\partial_{x_n}\sigma_{-2}(D^{-2})](x_0){\rm d}\xi_n\sigma(\xi'){\rm d}x'.
\end{align}
From Lemma 3.2, and direct computations, we obtain
\begin{eqnarray}\label{37}
\partial_{x_n}\sigma_{-2}(D^{-2})(x_0)|_{|\xi'|=1}
=-\frac{h'(0)}{(1+\xi_n^2)^2}.
\end{eqnarray}
 Then
 \begin{eqnarray}\label{37}
\partial_{\xi_n}\partial_{x_n}\sigma_{-2}(D^{-2})(x_0)|_{|\xi'|=1}
=\partial_{\xi_n}\big(-\frac{h'(0)}{(1+\xi_n^2)^2}\big)=\frac{4\xi_n h'(0)}{(1+\xi_n^2)^3}.
\end{eqnarray}
   By Lemma 3.4, we have
 \begin{align}
\label{b22}
\pi^+_{\xi_n}\Big(\sigma_{0}(-2c(w)\nabla_v^{S(TM)}D^{-1})\Big)
=&\sum_{j=1}^n\pi^+_{\xi_n}\Big((-2c(w))v_j\sqrt{-1}\xi_{j}\frac{\sqrt{-1}c(\xi)}{|\xi|^2}\Big)\nonumber\\
=&2\sum_{j=1}^{n-1}v_j\xi_{j}\Big(\pi^+_{\xi_n}\big(  \frac{c(w)c(\xi')}{1+\xi_n^2 } \big)
+\pi^+_{\xi_n}\big( \xi_n \frac{c(w)c(dx_{n})}{1+\xi_n^2 }  \big)
\Big)\nonumber\\
&+2 v_n\Big(\pi^+_{\xi_n}\big( \xi_n \frac{c(w)c(\xi')}{1+\xi_n^2 } \big)
+\pi^+_{\xi_n}\big( \xi_n^{2 }\frac{c(w)c(dx_{n})}{1+\xi_n^2 }  \big)
\Big).
\end{align}
Using the Cauchy integral formula, we obtain
  \begin{align}
\label{b22}
 \pi^+_{\xi_n}\big(  \frac{c(w)c(\xi')}{1+\xi_n^2 } \big)= \frac{-i}{2(\xi_n-i)}c(w)c(\xi'),
 \end{align}
  and
  \begin{align}
\label{b22}
 \pi^+_{\xi_n}\big( \xi_n^{2 }\frac{c(w)c(dx_{n})}{1+\xi_n^2 }  \big)= \frac{i}{2(\xi_n-i)}c(w)c(dx_{n}).
 \end{align}
It follows that
  \begin{align}
\label{b22}
\partial_{\xi_n} \pi^+_{\xi_n}\big(  \frac{c(w)c(\xi')}{1+\xi_n^2 } \big)= \partial_{\xi_n}\big(\frac{-i}{2(\xi_n-i)}\big)c(w)c(\xi')
=  \frac{i}{2(\xi_n-i)^{2}} c(w)c(\xi'),
 \end{align}
  and
  \begin{align}
\label{b22}
\partial_{\xi_n} \pi^+_{\xi_n}\big( \xi_n^{2 }\frac{c(w)c(dx_{n})}{1+\xi_n^2 }  \big)
= \partial_{\xi_n}\big(\frac{i}{2(\xi_n-i)}\big)c(w)c(dx_{n})=\frac{-i}{2(\xi_n-i)^{2}} c(w)c(dx_{n}).
 \end{align}
Combining (3.44)-(3.49), we obtain
\begin{align}\label{36}
\widetilde{\Phi^{*}}_3=&-\frac{1}{2}\int_{|\xi'|=1}\int^{+\infty}_{-\infty} \frac{4i \xi_{n} }{(\xi_{n}-i)^{5}(\xi_{n}+i)^{3}}
\sum_{j,k=1}^{n-1}(-v_{j}w_{k})\xi_{j}\xi_{k}{\rm Tr}_{S(TM)}[{\rm  id }](x_0) {\rm d}\xi_n\sigma(\xi'){\rm d}x'\nonumber\\
&-\frac{1}{2}\int_{|\xi'|=1}\int^{+\infty}_{-\infty} \frac{4i \xi_{n} }{(\xi_{n}-i)^{5}(\xi_{n}+i)^{3}}
v_{n}w_{n}  {\rm Tr}_{S(TM)}[{\rm  id }](x_0) {\rm d}\xi_n\sigma(\xi'){\rm d}x'\nonumber\\
=& -\frac{5}{6}\pi^{2}g^{M}(v^{T},w^{T})h'(0){\rm d}x' +\frac{5}{2}\pi^{2}v_{n}w_{n}h'(0){\rm d}x'.
\end{align}

 {\bf case a)~IV)}~$r=0,~l=-3,~k=j=|\alpha|=0$.

By (2.13), we get
\begin{align}
\widetilde{\Phi^{*}}_4&=-i\int_{|\xi'|=1}\int^{+\infty}_{-\infty}{\rm Tr} [\pi^+_{\xi_n}\sigma_{0}(-2c(w)\nabla_v^{S(TM)} D^{-1})\times
\partial_{\xi_n}\sigma_{-3}(D^{-2})](x_0){\rm d}\xi_n\sigma(\xi'){\rm d}x'\nonumber\\
&=i\int_{|\xi'|=1}\int^{+\infty}_{-\infty}{\rm Tr} [\partial_{\xi_n}\pi^+_{\xi_n}\sigma_{0}(-2c(w)\nabla_v^{S(TM)} D^{-1})\times
\sigma_{-3}(D^{-2})](x_0){\rm d}\xi_n\sigma(\xi'){\rm d}x'.
\end{align}
 By Lemma 3.5, we have
\begin{align}
\sigma_{-3}(D^{-2})(x_0)|_{|\xi'|=1}=-\frac{i}{(1+\xi_n^2)^2}\left(-\frac{1}{2}h'(0)\sum_{k<n}\xi_nc(e_k)c(e_n)
+\frac{5}{2}h'(0)\xi_n\right)-\frac{2ih'(0)\xi_n}{(1+\xi_n^2)^3}.
\end{align}
From Lemma 3.4 and direct computations, we obtain
 \begin{align}
\label{b22}
\partial_{\xi_n}\pi^+_{\xi_n}\Big(\sigma_{0}(-2c(w)\nabla_v^{S(TM)}D^{-1})\Big)
=&\sum_{j=1}^n\partial_{\xi_n}\pi^+_{\xi_n}\Big((-2c(w))v_j\sqrt{-1}\xi_{j}\frac{\sqrt{-1}c(\xi)}{|\xi|^2}\Big)\nonumber\\
=&2\sum_{j=1}^{n-1}v_j\xi_{j}\Big( \frac{i}{(\xi_n-i)^{2}} c(w)c(\xi')+ \frac{-1}{(\xi_n-i)^{2}} c(w)c(dx_{n})\Big)\nonumber\\
 &+v_{n}\Big(\frac{-1}{(\xi_n-i)^{2}} c(w)c(\xi')+ \frac{-i}{(\xi_n-i)^{2}} c(w)c(dx_{n})\Big).
\end{align}
By the relation of the Clifford action, we have the equalities:
\begin{align}
&\sum_{k<n}{\rm Tr} \Big(c(w)c(\xi')c(e_k)c(e_n)\Big)(x_0)\nonumber\\
=&\sum_{j,k=1}^{n-1}\sum_{l=1}^{n}w_{l}\xi_{j}\Big[{\rm Tr} \Big(c(e_l)c(e_j)c(e_k)c(e_n)\Big)\Big](x_0)\nonumber\\
=&\sum_{j,k=1}^{n-1}\sum_{l=1}^{n}w_{l}\xi_{j}\Big[{\rm Tr} \Big(-c(e_j)c(e_l)c(e_k)c(e_n)\Big)
-2\delta_{jl}{\rm Tr} \Big( c(e_k)c(e_n)\Big)\Big]\nonumber\\
=&\sum_{j,k=1}^{n-1}\sum_{l=1}^{n}w_{l}\xi_{j}\Big[{\rm Tr} \Big(c(e_j)c(e_k)c(e_l)c(e_n)\Big)
-2\delta_{kl}{\rm Tr} \Big( c(e_j)c(e_n)\Big)\Big]\nonumber\\
=&\sum_{j,k=1}^{n-1}\sum_{l=1}^{n}w_{l}\xi_{j}\Big[{\rm Tr} \Big(-c(e_j)c(e_k)c(e_n)c(e_l)\Big)
-2\delta_{ln}{\rm Tr} \Big( c(e_j)c(e_k)\Big)\Big],
\end{align}
then
\begin{align}
\sum_{k<n}{\rm Tr} \Big(c(w)c(\xi')c(e_k)c(e_n)\Big)
=\sum_{j=1}^{n-1}w_{n}\xi_{j}{\rm Tr}_{S(TM)}[{\rm  id }].
\end{align}
Combining (3.52)-(3.55), we get
\begin{align}
\label{b26}
\widetilde{\Phi^{*}}_4=& \int_{|\xi'|=1}\int^{+\infty}_{-\infty} \frac{1}{2(\xi_{n}-i)^{2}(1+\xi_{n}^{2})^{2}}
\sum_{j,k=1}^{n-1}\xi_{j}\xi_{k}(-v_{j}w_{j})h'(0){\rm Tr}_{S(TM)}[{\rm  id }](x_0){\rm d}\xi_n\sigma(\xi'){\rm d}x' \nonumber\\
&+ \int_{|\xi'|=1}\int^{+\infty}_{-\infty}   \frac{1}{2(\xi_{n}-i)^{2}(1+\xi_{n}^{2})^{2}}
\sum_{j,k=1}^{n-1}\xi_{j}\xi_{k}v_{n}w_{n}h'(0){\rm Tr}_{S(TM)}[{\rm  id }] (x_0){\rm d}\xi_n\sigma(\xi'){\rm d}x' \nonumber\\
&- \int_{|\xi'|=1}\int^{+\infty}_{-\infty} \frac{9i\xi_{n}+5i\xi_{n}^{2}}{2(\xi_{n}-i)^{2}(1+\xi_{n}^{2})^{2}}
\sum_{j,k=1}^{n-1}\xi_{j}\xi_{k}v_{j}w_{k}h'(0)
{\rm Tr}_{S(TM)}[{\rm  id }](x_0){\rm d}\xi_n\sigma(\xi'){\rm d}x' \nonumber\\
&+\int_{|\xi'|=1}\int^{+\infty}_{-\infty}   \frac{9i\xi_{n}+5i\xi_{n}^{2}}{2(\xi_{n}-i)^{2}(1+\xi_{n}^{2})^{2}}
\sum_{j,k=1}^{n-1}\xi_{j}\xi_{k}v_{n}w_{n}h'(0){\rm Tr}_{S(TM)}[{\rm  id }] (x_0){\rm d}\xi_n\sigma(\xi'){\rm d}x' \nonumber\\
=& \frac{19}{6}\pi^{2}g^{M}(v^{T},w^{T})h'(0){\rm d}x' -\frac{49}{6}\pi^{2}v_{n}w_{n}h'(0){\rm d}x'.
 \end{align}

 {\bf case a)~V)}~$r=-1,~\ell=-2,~k=j=|\alpha|=0$.

By  (2.13), we get
\begin{align}\label{61}
\widetilde{\Phi^{*}}_5=-i\int_{|\xi'|=1}\int^{+\infty}_{-\infty}{\rm Tr} [\pi^+_{\xi_n}\sigma_{-1}(-2c(w)\nabla_v^{S(TM)} D^{-1})\times
\partial_{\xi_n}\sigma_{-2}(D^{-2})](x_0){\rm d}\xi_n\sigma(\xi'){\rm d}x'.
\end{align}
By Lemma 3.2, then
\begin{align}\label{62}
\partial_{\xi_n}\sigma_{-2}(D^{-2})|_{|\xi'|=1}(x_0)=-\frac{2\xi_n}{(\xi_n^2+1)^2}.
\end{align}
From Lemma 3.4, we obtain
\begin{align}
\label{b22}
\sigma_{-1}(-2c(w)\nabla_v^{S(TM)}D^{-1})=&\sigma_{1}(-2c(w)\nabla_v^{S(TM)} )\sigma_{-2}(D^{-1})
+\sigma_{0}(-2c(w)\nabla_v^{S(TM)} )\sigma_{-1}(D^{-1})\nonumber\\
&+\sum_{j=1}^{n}\partial_{\xi_{j}}\sigma_{1}\big(-2c(w)\nabla_v^{S(TM)} \big)
D_{x_{j}}\big(\sigma_{-1}(D^{-1})\big)\nonumber\\
=:&H_{1}+H_{2}+H_{3}.
\end{align}

(1) Explicit representation of $H_{1}$:

By Lemma 3.5, we have
\begin{align}
&\sigma_{-2}(D^{-1})(x_{0})=\frac{c(\xi)(-\frac{3}{4}h'(0)c(dx_n))c(\xi)}{|\xi|^4}+\frac{c(\xi)}{|\xi|^6}c(dx_n)
[\partial_{x_n}[c(\xi)(x_{0})]|\xi|^2-c(\xi)h'(0)|\xi'|^2)].
\end{align}
A simple computation shows
\begin{align}
\label{b22}
H_{1}(x_0)=&\sigma_{-1}(-2c(w)\nabla_v^{S(TM)} )\sigma_{-2}(D^{-1})\nonumber\\
=&\sum_{j=1}^nv_j(-2c(w))\sqrt{-1}\xi_{j}\times
\Big(\frac{c(\xi)(-\frac{3}{4}h'(0)c(dx_n))c(\xi)}{|\xi|^4}\nonumber\\
&+\frac{c(\xi)}{|\xi|^6}c(dx_n)
[\partial_{x_n}[c(\xi)(x_{0})]|\xi|^2-c(\xi)h'(0)|\xi'|^2)]\Big)\nonumber\\
=&\sum_{j=1}^{n-1}v_j(-2c(w))\sqrt{-1}\xi_{j}
\Big(\frac{3\xi_{n}^{4}+4\xi_{n}^{2}-7}{4(1+\xi_{n}^{2})^{3}}h'(0)c(dx_n)
+ \frac{3\xi_{n}^{3}+7\xi_{n} }{ (1+\xi_{n}^{2})^{3}}h'(0)c(\xi')\nonumber\\
&+\frac{1}{(1+\xi_{n}^{2})^{2}} c(\xi')c(dx_n)\partial_{x_n}[c(\xi)]
+\frac{-\xi_{n}}{(1+\xi_{n}^{2})^{2}} \partial_{x_n}[c(\xi)]
\Big)\nonumber\\
+&v_n(-2c(w))\sqrt{-1}\xi_{n}
\Big(\frac{3\xi_{n}^{4}+4\xi_{n}^{2}-7}{4(1+\xi_{n}^{2})^{3}}h'(0)c(dx_n)
+ \frac{3\xi_{n}^{3}+7\xi_{n} }{ (1+\xi_{n}^{2})^{3}}h'(0)c(\xi')\nonumber\\
&+\frac{1}{(1+\xi_{n}^{2})^{2}} c(\xi')c(dx_n)\partial_{x_n}[c(\xi)]
+\frac{-\xi_{n}}{(1+\xi_{n}^{2})^{2}} \partial_{x_n}[c(\xi)]\Big).
\end{align}
By  the Cauchy integral formula, then
\begin{align}
 \pi^+_{\xi_n}\left[\frac{-i\xi_{n}(3\xi_{n}^{4}+4\xi_{n}^{2}-7)}{2(1+\xi_{n}^{2})^{3}}\right]
 =&\frac{1}{2\pi i}\lim_{u\rightarrow 0^{-}}\int_{\Gamma^+}\frac{\frac{-i\eta_{n}(3\eta_{n}^{4}+4\eta_{n}^{2}-7)}
 {2(\eta_n+i)^3(\xi_n+iu-\eta_n)}}{(\eta_n-i)^3}
 d\eta_n\nonumber\\
 =&\frac{1}{2\pi i} \frac{2\pi i}{2£¡}\left[\frac{-i\eta_{n}(3\eta_{n}^{4}+4\eta_{n}^{2}-7}
 {2(\eta_n+i)^3(\xi_n-\eta_n)}\right]^{(2)}\Big|_{\eta_n=i}\nonumber\\
 =&\frac{-i}{2(\xi_n-i)^3}.
\end{align}
Similarly,
\begin{align}
\pi^+_{\xi_n}\left[\frac{-2i\xi_{n}}{(1+\xi_{n}^{2})^{2}}\right](x_0)|_{|\xi'|=1}
 =\frac{1}{2(\xi_n-i)^2}.
\end{align}
Substituting (3.60) and (3.61) into (3.57), we obtain
\begin{align}\label{61}
&-i\int_{|\xi'|=1}\int^{+\infty}_{-\infty}{\rm Tr} [\pi^+_{\xi_n}(H_{1})\times
\partial_{\xi_n}\sigma_{-2}(D^{-2})](x_0){\rm d}\xi_n\sigma(\xi'){\rm d}x'\nonumber\\
=& \int_{|\xi'|=1}\int^{+\infty}_{-\infty} \frac{2\xi_{n}}{(\xi_{n}-i)^{3}(1+\xi_{n}^{2})^{2}}
\sum_{j,k=1}^{n-1}\xi_{j}\xi_{k}v_{j}w_{j}h'(0){\rm Tr}_{S(TM)}[{\rm  id }](x_0){\rm d}\xi_n\sigma(\xi'){\rm d}x' \nonumber\\
&+ \int_{|\xi'|=1}\int^{+\infty}_{-\infty}   \frac{i\xi_{n}}{2(\xi_{n}-i)^{2}(1+\xi_{n}^{2})^{2}}
\sum_{j,k=1}^{n-1}\xi_{j}\xi_{k}v_{j}w_{k}h'(0){\rm Tr}_{S(TM)}[{\rm  id }] (x_0){\rm d}\xi_n\sigma(\xi'){\rm d}x' \nonumber\\
&+\int_{|\xi'|=1}\int^{+\infty}_{-\infty} \frac{-\xi_{n}}{ (\xi_{n}-i)^{3}(1+\xi_{n}^{2})^{2}}
 v_{n}w_{n}h'(0)
{\rm Tr}_{S(TM)}[{\rm  id }](x_0){\rm d}\xi_n\sigma(\xi'){\rm d}x' \nonumber\\
&+\int_{|\xi'|=1}\int^{+\infty}_{-\infty}   \frac{i\xi_{n}}{ 2(\xi_{n}-i)^{2}(1+\xi_{n}^{2})^{2}}
 v_{n}w_{n}h'(0){\rm Tr}_{S(TM)}[{\rm  id }] (x_0){\rm d}\xi_n\sigma(\xi'){\rm d}x' \nonumber\\
=& -\frac{4}{3}\pi^{2}g(v^{T},w^{T})h'(0)dx' +\frac{1}{2}\pi^{2}v_{n}w_{n}h'(0){\rm d}x'.
\end{align}

(2) Explicit representation of $H_{2}$:

let $A(v)=\frac{1}{4}\Sigma_{ij} \langle\nabla_v^L{\tilde{e}_i},\tilde{e}_j\rangle c(\tilde{e}_i)c(\tilde{e}_j)$, by Lemma 3.4, we obtain
\begin{align}
\label{b22}
H_{2} =&\sigma_{0}(-2c(w)\nabla_v^{S(TM)} )\sigma_{-1}(D^{-1})
    =-2c(w)A(v)\frac{\sqrt{-1}c(\xi)}{|\xi|^2}.
\end{align}
By the Cauchy integral formula, then
\begin{align}
 \pi^+_{\xi_n}(H_{2})=& \pi^+_{\xi_n}\Big(-2c(w)A(v)\frac{\sqrt{-1}c(\xi)}{|\xi|^2}\Big)\nonumber\\
   =& \frac{-1}{  \xi_n-i }c(w)A(v)c(\xi')+\frac{-1}{  \xi_n-i }c(w)A(v)c(dx_{n}).
\end{align}
 We note  that $i<n,~\int_{|\xi'|=1}\xi_{i_{1}}\xi_{i_{2}}\cdots\xi_{i_{2d+1}}\sigma(\xi')=0$,
so we omit some items that have no contribution for computing  {\bf case a)~V)}. By the relation of the Clifford action, then
we have the equalities
\begin{align}
&{\rm Tr} \Big(c(w)A(v)c(dx_{n})\Big)\nonumber\\
=&{\rm Tr} \Big(c(w)\frac{1}{4} \Sigma_{ij} \langle\nabla_v^L{\tilde{e}_i},\tilde{e}_j\rangle c(\tilde{e}_i)c(\tilde{e}_j)c(dx_{n})\Big)\nonumber\\
=&\frac{1}{4}\sum_{i,j=1}^{n}\langle\nabla_v^L{\tilde{e}_i},\tilde{e}_j\rangle{\rm Tr} \Big(c(w) c(\tilde{e}_i)c(\tilde{e}_j)c(dx_{n})\Big)\nonumber\\
=&\frac{1}{4}\sum_{i,j,l=1}^{n}\langle\nabla_v^L{\tilde{e}_i},\tilde{e}_j\rangle w_{l}
\Big[{\rm Tr} \Big( -c(\tilde{e}_i)c(e_l)c(\tilde{e}_j)c(dx_{n})\Big) -2\delta_{ il}{\rm Tr}(c(\tilde{e}_j)c(dx_{n}))\Big]\nonumber\\
=&\frac{1}{4}\sum_{i,j,l=1}^{n}\langle\nabla_v^L{\tilde{e}_i},\tilde{e}_j\rangle w_{l}
\Big[{\rm Tr} \Big( c(\tilde{e}_i)c(\tilde{e}_j) c(e_l)c(dx_{n})\Big) -2\delta_{ jl}{\rm Tr}(c(\tilde{e}_i)c(dx_{n}))
+2\delta_{ il}\delta_{ jn}{\rm Tr}_{S(TM)}[{\rm \texttt{id}}] \Big]\nonumber\\
=&\frac{1}{4}\sum_{i,j,l=1}^{n}\langle\nabla_v^L{\tilde{e}_i},\tilde{e}_j\rangle w_{l}
\Big[{\rm Tr} \Big( -c(\tilde{e}_i)c(\tilde{e}_j)c(dx_{n})c(e_l))\Big) -2\delta_{ln}{\rm Tr}(c(\tilde{e}_i)c(dx_{j}))\nonumber\\
&+2\delta_{ il}\delta_{ jn}{\rm Tr}_{S(TM)}[{\rm \texttt{id}}]
+2\delta_{ jl}\delta_{ in}{\rm Tr}_{S(TM)}[{\rm \texttt{id}}]
 \Big],
\end{align}
and
\begin{align}
 {\rm Tr} \Big(c(w)A(v)c(dx_{n})\Big)
=&\frac{1}{8}\sum_{i,j,l=1}^{n}\langle\nabla_v^L{e_i},e_j\rangle w_{l} \Big(\delta_{ln}\delta_{ij}
 +2\delta_{ il}\delta_{ jn} +2\delta_{ jl}\delta_{ in}\Big){\rm Tr}_{S(TM)}[{\rm  id }]\nonumber\\
  =&\frac{1}{8}\Big(\sum_{i,j,l=1}^{n}\langle\nabla_v^L{e_i},e_j\rangle w_{l}
 \delta_{ln}\delta_{ij}
 + \sum_{i,j,l=1}^{n}\langle\nabla_v^L{e_i},e_j\rangle w_{l}\delta_{ il}\delta_{ jn} \nonumber\\
& +\sum_{i,j,l=1}^{n}\langle\nabla_v^L{e_i},e_j\rangle w_{l}  \delta_{ jl}\delta_{ in}\Big){\rm Tr}_{S(TM)}[{\rm  id }]\nonumber\\
 =&\frac{1}{8}\Big(\sum_{i =1}^{n}\langle\nabla_v^L{e_i},e_i \rangle w_{n}
  + \sum_{j=1}^{n}\langle\nabla_v^L{e_n},e_j \rangle w_{j}  \nonumber\\
& +\sum_{i=1}^{n}\langle\nabla_v^L{e_i},e_n\rangle w_{i}   \Big){\rm Tr}_{S(TM)}[{\rm id}]\nonumber\\
=&0,
 \end{align}
where
$
\langle\nabla_v^L{e_i},e_i \rangle+\langle e_i ,\nabla_v^L{e_i}\rangle=\nabla_v^L\langle e_i,e_i \rangle=0,
$
then
$
\langle\nabla_v^L{e_i},e_i \rangle=0.
$
And a simple computation shows
\begin{align}
\sum_{i=1}^{n}\langle\nabla_v^L{e_i},e_n \rangle w_{i}
=-\sum_{i=1}^{n-1} \langle e_i,\nabla_v^L{e_n} \rangle w_{i}= -\langle\nabla_v^L{e_n},w\rangle.
 \end{align}
Similarly, $\sum_{i=1}^{n}\langle\nabla_v^L{e_n},e_j \rangle w_{j}=\langle\nabla_v^L{e_n},w\rangle$.
Then adding these identities gives
\begin{align}\label{61}
 -i\int_{|\xi'|=1}\int^{+\infty}_{-\infty}{\rm Tr} [\pi^+_{\xi_n}(H_{2})\times
\partial_{\xi_n}\sigma_{-2}(D^{-2})](x_0){\rm d}\xi_n\sigma(\xi'){\rm d}x'=0.
\end{align}

(3)Explicit representation of $H_{3}$:
\begin{align}
\label{b22}
H_{3}=&\sum_{j=1}^{n}\partial_{\xi_{j}}\sigma_{1}\big(-2c(w)\nabla_v^{S(TM)} \big)
D_{x_{j}}\big(\sigma_{-1}(D^{-1})\big)(x_{0})\nonumber\\
=&-2c(w)v_{n}\Big( \frac{\sqrt{-1}\partial_{x_n}[c(\xi')]}{|\xi|^2}-\frac{\sqrt{-1}c(\xi)h'(0)|\xi'|^2
}{|\xi|^4}
\Big).
\end{align}
By the Cauchy integral formula, we obtain
\begin{align}
 \pi^+_{\xi_n}(H_{3})=& \pi^+_{\xi_n}\Big[ -2c(w)v_{n}\Big( \frac{\sqrt{-1}\partial_{x_n}[c(\xi')]}{|\xi|^2}-\frac{\sqrt{-1}c(\xi)h'(0)|\xi'|^2
}{|\xi|^4}
\Big)\Big]\nonumber\\
   =& \frac{-1}{  \xi_n-i }v_{n}c(w)\partial_{x_n}[c(\xi')]+\frac{1}{  2(\xi_n-i)^{2} }h'(0)v_{n}c(w)c(dx_{n})\nonumber\\
   &++\frac{\xi_n-2i}{  2(\xi_n-i)^{2} }h'(0)v_{n}c(w)c(\xi').
\end{align}
Substituting (3.60) and (3.72) into (3.57), we obtain
\begin{align}\label{61}
&-i\int_{|\xi'|=1}\int^{+\infty}_{-\infty}{\rm Tr} [\pi^+_{\xi_n}(H_{3})\times
\partial_{\xi_n}\sigma_{-2}(D^{-2})](x_0){\rm d}\xi_n\sigma(\xi'){\rm d}x'\nonumber\\
=& \int_{|\xi'|=1}\int^{+\infty}_{-\infty} \frac{-i\xi_{n}}{(\xi_{n}-i)^{4}(\xi_{n}+i)^{2}}
v_{n}w_{n}h'(0){\rm Tr}_{S(TM)}[{\rm \texttt{id}}](x_0){\rm d}\xi_n\sigma(\xi'){\rm d}x' \nonumber\\
=& 2 \pi^{2}v_{n}w_{n}h'(0){\rm d}x'.
\end{align}
Therefore, we get
\begin{align}
\widetilde{\Phi^{*}}_5= -\frac{4}{3}\pi^{2}g(v^{T},w^{T})h'(0){\rm d}x' +\frac{5}{2}\pi^{2}v_{n}w_{n}h'(0){\rm d}x'.
\end{align}

 Now $\widetilde{\Phi^{*}}$ is the sum of the cases  {\bf case a)~I)}-- {\bf case a)~V)}. Combining with the five cases, this yields
\begin{align}\label{795}
\widetilde{\Phi^{*}}=&  \frac{2}{3}\pi^{2} \partial_{x_n}\big(g(v^{T},w^{T})\big){\rm d}x'
  + \frac{1}{3}\pi^{2}\partial_{x_n}\big(v_{n}w_{n}\big)dx'\nonumber\\
 &-\frac{7}{6}\pi^{2}g(v^{T},w^{T})h'(0)dx' -\frac{11}{3}\pi^{2}v_{n}w_{n}h'(0){\rm d}x'.
\end{align}

Recall the Einstein-Hilbert action for manifolds with boundary \cite{Wa3},
\begin{align}
I_{\rm Gr}=\frac{1}{16\pi}\int_Ms{\rm dvol}_M+2\int_{\partial M}K{\rm dvol}_{\partial_M}:=I_{\rm {Gr,i}}+I_{\rm {Gr,b}},
\end{align}
 where
\begin{align}
K=\sum_{1\leq i,j\leq {n-1}}K_{i,j}g_{\partial M}^{i,j};~~K_{i,j}=-\Gamma^n_{i,j},
\end{align}
 and $K_{i,j}$ is the second fundamental form, or extrinsic
curvature. Take the metric in Section 2, then by Lemma A.2\cite{Wa3},
$K_{i,j}(x_0)=-\Gamma^n_{i,j}(x_0)=-\frac{1}{2}h'(0),$ when $i=j<n$,
otherwise is zero. For $n=4$, then
\begin{align}
K(x_0)=\sum_{i,j}K_{i.j}(x_0)g_{\partial M}^{i,j}(x_0)=\sum_{i=1}^3K_{i,i}(x_0)=-\frac{3}{2}h'(0).
\end{align}
Substituting (3.74) into (3.71), we have
 \begin{align}\label{795}
\widetilde{\Phi^{*}}=&  \frac{2}{3}\pi^{2} \partial_{x_n}g^{M}(v^{T},w^{T}\big){\rm d}x'
  + 2\pi^{2}\partial_{x_n}\big(v_{n}w_{n}\big){\rm d}x'\nonumber\\
 &-\frac{7}{9}\pi^{2}g^{M}(v^{T},w^{T}) K(x_0) {\rm d}x'+\frac{22}{9}\pi^{2}v_{n}w_{n}K(x_0){\rm d}x'.
\end{align}

Using an explicit formula for $\Phi^{*}$ and $\widetilde{\Phi^{*}}$, we can
reformulate Proposition 2.4 as follows.
\begin{thm}\label{thmb1}
Let $M$ be a $4$-dimensional oriented
compact spin manifold with boundary $\partial M$ and the metric
$g^{M}$ be defined as above, then we get the following equality:
\begin{align}
\label{b2773}
&{\rm{\widetilde{Wres}}}\Big(\pi^+\big(c(w)(Dc(v)+c(v)D)D^{-1}\big)\circ \pi^+(D^{-2})\Big)\nonumber\\
=&\frac{4\pi^{2}}{3}\int_{M}(Ric(v,w)-\frac{1}{2}s g(v,w))vol_{g}
+\int_{\partial M}\Big[\frac{2}{3}\pi^{2} \partial_{x_n}g(v^{T},w^{T}\big)dx'
  + 2\pi^{2}\partial_{x_n}\big(v_{n}w_{n}\big)dx'\nonumber\\
   &-2\pi^2 \Big(\sum_{j=1}^{n}g(e_{j},\nabla^{ TM}_{e_{j}}v)g(w,e_{n})-g(w,\nabla^{ TM}_{e_{n}}v)+g(\nabla^{ TM}_{w}v, e_{n}) \Big) \nonumber\\
 &-\frac{7}{9}\pi^{2}g^{M}(v^{T},w^{T}) K(x_0) dx'+\frac{22}{9}\pi^{2}v_{n}w_{n}K(x_0)\Big]{\rm d}x'.
\end{align}
\end{thm}

\section{Residue for type-II operator}
The aim of this section is to prove the the Einstein functional for four dimensional  manifold with boundary, and
get the Dabrowski-Sitarz-Zalecki type theorems in this case.
\begin{prop}
For the   type-II operator, the Einstein functional for four dimensional  manifold with boundary defined by
\begin{align}
{\rm{\widetilde{Wres}}}\Big(\pi^+\big(c(w)(Dc(v)+c(v)D)D^{-2}\big)\circ \pi^+(D^{-1})\Big)
=&{\rm{\widetilde{Wres}}}\Big(\pi^+\big(\sum_{j=1}^{n}c(w)c(e_{j})c(\nabla^{TM}_{e_{j}}v)D^{-2}\big)\circ \pi^+(D^{-1})\Big)\nonumber\\
&+{\rm{\widetilde{Wres}}}\Big(\pi^+\big(-2c(w)\nabla^{ S(TM)}_{v}D^{-2}\big)\circ \pi^+(D^{-1})\Big).
\end{align}
\end{prop}

\begin{proof}
By straightforward computations yield
\begin{align}
c(w)\big(Dc(v)+c(v)D\big)D^{-2}&=c(w)Dc(v)D^{-1}D^{-1}+c(w)c(v)D^{-1}\nonumber\\
=&\Big(\sum_{j=1}^{n}c(w)c(e_{j})c(\nabla^{TM}_{e_{j}}v)D^{-1}-c(w)c(v)-2c(w)\nabla^{ S(TM)}_{v}D^{-1}\Big)D^{-1}\nonumber\\
  &+c(w)c(v)D^{-1}\nonumber\\
=&\sum_{j=1}^{n}c(w)c(e_{j})c(\nabla^{TM}_{e_{j}}v)D^{-2}-2c(w)\nabla^{ S(TM)}_{v}D^{-2}.
\end{align}
\end{proof}
Combining with the generating Proposition 2.5, this yields
\begin{align}
&{\rm{\widetilde{Wres}}}\Big(\pi^+\big(\sum_{j=1}^{n}c(w)c(e_{j})c(\nabla^{TM}_{e_{j}}v)D^{-2}\big)\circ \pi^+(D^{-1})\Big)\nonumber\\
=&Wres\big( \sum_{j=1}^{n}c(w)c(e_{j})c(\nabla^{TM}_{e_{j}}v)D^{-3}\big) +\int_{\partial M}\Phi^{**},
\end{align}
and
\begin{align}
 {\rm{\widetilde{Wres}}}\Big(\pi^+\big(-2c(w)\nabla^{ S(TM)}_{v}D^{-2}\big)\circ \pi^+(D^{-1})\Big)
= Wres \big(-2c(w)\nabla^{ S(TM)}_{v}D^{-3}\big) +\int_{\partial M}\widetilde{\Phi^{**}}.
\end{align}

\subsection{The computation of $\Phi^{**}$}

Now we can compute $\Phi^{**}$, since the sum is taken over $
-r-l+k+j+|\alpha|=3,~~r\leq-2,l\leq-1$, we get $r=-2,l=-1,~k=|\alpha|=j=0,$ then
\begin{align}
\Phi^{**}=-i\int_{|\xi'|=1}\int^{+\infty}_{-\infty}{\rm Tr} \Big[\pi^+_{\xi_n}\sigma_{-2}
\Big(\sum_{j=1}^{n}c(w)c(e_{j})c(\nabla^{ TM}_{e_{j}}v)D^{-2}\Big)\times
\partial_{\xi_n}\sigma_{-1}(D^{-1})\Big](x_0){\rm d}\xi_n\sigma(\xi'){\rm d}x'.
\end{align}
By Lemma 3.2 and the Cauchy integral formula, then
\begin{align}
\pi^+_{\xi_n}\sigma_{-2}\Big(\sum_{j=1}^{n}c(w)c(e_{j})c(\nabla^{ TM}_{e_{j}}v)D^{-2}\Big)
=\frac{1}{2i(\xi_n-i)}\sum_{j=1}^{n}c(w)c(e_{j})c(\nabla^{ TM}_{e_{j}}v).
\end{align}
From Lemma 3.2 and direct computations, we obtain
\begin{align}
\partial_{\xi_n}\sigma_{-1}(D^{-1})(x_0)|_{|\xi'|=1}=\sqrt{-1}
\left[\frac{c(dx_n)}{1+\xi_n^2}-\frac{2\xi_nc(\xi')+2\xi_n^2c(dx_n)}{(1+\xi_n^2)^2}\right].
\end{align}
A simple computation shows
\begin{align}
&{\rm Tr} \Big[\pi^+_{\xi_n}\sigma_{-1}\Big(\sum_{j=1}^{n}c(w)c(e_{j})c(\nabla^{ TM}_{e_{j}}v)D^{-1}\Big)\times
\partial_{\xi_n}\sigma_{-2}(D^{-2})\Big](x_0)\nonumber\\
=& \frac{-\xi_n}{(\xi_n-i)^3(\xi_n+i)^2}{\rm Tr} \Big(\sum_{j=1}^{n}c(w)c(e_{j})c(\nabla^{ TM}_{e_{j}}v)c(\xi')\Big)\nonumber\\
&+\frac{1-\xi_n^{2}}{2(\xi_n-i)^3(\xi_n+i)^2} {\rm Tr} \Big(\sum_{j=1}^{n}c(w)c(e_{j})c(\nabla^{ TM}_{e_{j}}v)c(dx_{n})\Big).
\end{align}
Combining (4.6), (4.7) and (4.8), we obtain
\begin{align}
\Phi^{**}=&-i\int_{|\xi'|=1}\int^{+\infty}_{-\infty} \Big[\frac{1-\xi_n^{2}}{2(\xi_n-i)^3(\xi_n+i)^2} {\rm Tr}
 \Big(\sum_{j=1}^{n}c(w)c(e_{j})c(\nabla^{ TM}_{e_{j}}v)c(dx_{n})\Big)\Big](x_0){\rm d}\xi_n\sigma(\xi'){\rm d}x'\nonumber\\
  =&\int_{\Gamma^{+}}\frac{i\xi_n^{2}-i}{2(\xi_n-i)^3(\xi_n+i)^2}{\rm Tr}
 \Big(\sum_{j=1}^{n}c(w)c(e_{j})c(\nabla^{ TM}_{e_{j}}v)c(dx_{n})\Big)\Big](x_0)\Omega_3 {\rm d}x'\nonumber\\
 =&2\pi^2 \Big(\sum_{j=1}^{n}g(e_{j},\nabla^{ TM}_{e_{j}}v)g(w,e_{n})-g(w,\nabla^{ TM}_{e_{n}}v)+g(\nabla^{ TM}_{w}v, e_{n}) \Big){\rm d}x'.
\end{align}

\subsection{The computation of $\widetilde{\Phi^{**}}$}

Now we  need to compute $\int_{\partial M} \widetilde{\Phi^{**}}$. When $n=4$, then ${\rm Tr}_{S(TM)}[{\rm \texttt{id}}]
={\rm dim}(\wedge^*(\mathbb{R}^2))=4$,  the sum is taken over $r+l-k-j-|\alpha|=-3,~~r\leq -1,~~l\leq-1,$ then we have the following five cases:

{\bf case b)~I)}~$r=-1,~l=-1,~k=j=0,~|\alpha|=1$

  By (2.15), we get
  \begin{align}
\widetilde{\Phi^{**}}_{1}=-\int_{|\xi'|=1}\int^{+\infty}_{-\infty}\sum_{|\alpha|=1}
{\rm Tr} [\partial^\alpha_{\xi'}\pi^+_{\xi_n} \sigma_{-1}(-2c(w)\nabla^{ S(TM)}_{v}D^{-2})     \times
\partial^\alpha_{x'}\partial_{\xi_n}\sigma_{-1}(D^{-1})](x_0){\rm d}\xi_n\sigma(\xi'){\rm d}x'.
\end{align}
From Lemma 3.2 and Lemma 3.6, for $i<n$, then
 \begin{align}
\partial_{x_i}\sigma_{-1}(D^{-1})(x_0)=\partial_{x_i}\left(\frac{\sqrt{-1}c(\xi)}{|\xi|^2}\right)(x_0)=
\frac{\sqrt{-1}\partial_{x_i}[c(\xi)](x_0)}{|\xi|^2}
-\frac{\sqrt{-1}c(\xi)\partial_{x_i}(|\xi|^2)(x_0)}{|\xi|^4}=0,
\end{align}
\noindent so case b) I) vanishes.

{\bf case b)~II)}~$r=-1,~l=-1,~k=|\alpha|=0,~j=1$

  By (2.15), we get
 \begin{align}
\widetilde{\Phi^{**}}_{2}=-\frac{1}{2}\int_{|\xi'|=1}\int^{+\infty}_{-\infty} {\rm
Tr} [\partial_{x_n}\pi^+_{\xi_n}\sigma_{-1}(-2c(w)\nabla^{ S(TM)}_{v}D^{-2}) \times
\partial_{\xi_n}^2\sigma_{-1}(D^{-1})](x_0){\rm d}\xi_n\sigma(\xi'){\rm d}x'.
\end{align}
From Lemma 3.2 and direct computations, we obtain
 \begin{align}
\partial^2_{\xi_n}\sigma_{-1}(D^{-1})=\sqrt{-1}\left(-\frac{6\xi_nc(dx_n)+2c(\xi')}
{|\xi|^4}+\frac{8\xi_n^2c(\xi)}{|\xi|^6}\right).
\end{align}
 By Lemma 3.4, then
 \begin{align}
\label{b22}
\sigma_{-1}(-2c(w)\nabla_v^{S(TM)}D^{-2})=&\sum_{j=1}^n(-2c(w))v_j\sqrt{-1}\xi_{j}|\xi|^{-2}.
\end{align}
An easy calculation gives
 \begin{align}
 \partial_{x_n}\sigma_{-1}(-2c(w)\nabla_v^{S(TM)}D^{-2})
=&\partial_{x_n}\Big(\sum_{j=1}^n(-2c(w))v_j\sqrt{-1}\xi_{j}|\xi|^{-2}\Big)\nonumber\\
=&-2\sqrt{-1}\sum_{j=1}^n\xi_{j}\Big[\partial_{x_n}\big(v_j c(w)\big)|\xi|^{-2}
  +v_j c(w)\partial_{x_n}\Big(|\xi|^{-2} \Big)\Big].
  \end{align}
By the Cauchy integral formula, then
 \begin{align}
\pi^+_{\xi_n} \partial_{x_n}\sigma_{-1}(-2c(w)\nabla_v^{S(TM)}D^{-2})
=&\frac{-1}{\xi_n-i}\sum_{j=1}^{n-1}\xi_j \partial_{x_n}(v_{j}c(w))
+\frac{-i(2+i\xi_n)}{2(\xi_n-i)^{2}}\sum_{j=1}^{n-1}\xi_j h'(0)v_{j}c(w)\nonumber\\
&+\frac{-i}{\xi_n-i}\partial_{x_n}(v_{n}c(w))+\frac{1}{2(\xi_n-i)^{2}}h'(0)v_{n}c(w).
  \end{align}
Substituting (4.13) (4.16) into (4.12), we have
  \begin{align}
\label{b26}
\widetilde{\Phi^{**}}_2=&-\frac{1}{2}\int_{|\xi'|=1}\int^{+\infty}_{-\infty}
{\rm Tr} \big[\partial_{x_n}\pi^+_{\xi_n}\sigma_{-1}(-2c(w)\nabla^{ S(TM)}_{v}D^{-2}) \times
\partial_{\xi_n}^2\sigma_{-1}(D^{-1})\big](x_0){\rm d}\xi_n\sigma(\xi'){\rm d}x' \nonumber\\
=&-\frac{1}{2}\int_{|\xi'|=1}\int^{+\infty}_{-\infty} \frac{i(2-6\xi_{n}^{2})}{(\xi_{n}-i)(1+\xi_{n}^{2})^{3}}
\sum_{j=1}^{n-1}\xi_{j}{\rm Tr}\Big( \partial_{x_n}\big(v_{j}c(w)\big)c(\xi') \Big)(x_0){\rm d}\xi_n\sigma(\xi'){\rm d}x' \nonumber\\
&-\frac{1}{2}\int_{|\xi'|=1}\int^{+\infty}_{-\infty} \frac{(2+i\xi_{n})(3\xi_{n}^{2}-1)}{(\xi_{n}-i)^{2}(1+\xi_{n}^{2})^{3}}
\sum_{j=1}^{n-1}\xi_{j}h'(0){\rm Tr}\big( v_{j}c(w)c(\xi')\big) (x_0){\rm d}\xi_n\sigma(\xi'){\rm d}x' \nonumber\\
&-\frac{1}{2}\int_{|\xi'|=1}\int^{+\infty}_{-\infty} \frac{ 2\xi_{n}^{3}-6\xi_{n}}{(\xi_{n}-i)(1+\xi_{n}^{2})^{3}}
{\rm Tr}\Big( \partial_{x_n}\big(v_{n}c(w)\big)c(dx_{n}) \Big)(x_0){\rm d}\xi_n\sigma(\xi'){\rm d}x' \nonumber\\
&-\frac{1}{2}\int_{|\xi'|=1}\int^{+\infty}_{-\infty} \frac{i (\xi_{n}^{3}-3\xi_{n})}{(\xi_{n}-i)^{2}(1+\xi_{n}^{2})^{3}}
h'(0){\rm Tr}\Big( \partial_{x_n}\big(v_{n}c(w)c(dx_{n})\big) \Big)(x_0){\rm d}\xi_n\sigma(\xi'){\rm d}x' \nonumber\\
=&-\frac{2}{3}\pi^{2} \partial_{x_n}\big(g^{M}(v^{T},w^{T})\big){\rm d}x'
  +\frac{5}{6}\pi^{2}g^{M}(v^{T},w^{T}) h'(0){\rm d}x'\nonumber\\
  &-2\pi^{2}\partial_{x_n}\big(v_{n}w_{n}\big)dx'-\frac{3}{2}\pi^{2}v_{n}w_{n}h'(0){\rm d}x'.
 \end{align}

{\bf case b)~III)}~$r=-1,~l=-1,~j=|\alpha|=0,~k=1$

  By (2.15), we get
 \begin{align}
\widetilde{\Phi^{**}}_{3}=-\frac{1}{2}\int_{|\xi'|=1}\int^{+\infty}_{-\infty}
{\rm Tr} [\partial_{\xi_n}\pi^+_{\xi_n}\sigma_{-1}(-2c(w)\nabla^{ S(TM)}_{v}D^{-2})\times
\partial_{\xi_n}\partial_{x_n}\sigma_{-1}(D^{-1})](x_0){\rm d}\xi_n\sigma(\xi'){\rm d}x'.
  \end{align}
 By Lemma 3.2, we have
 \begin{align}
\partial_{\xi_n}\partial_{x_n}\sigma_{-1}(D^{-1})(x_0)|_{|\xi'|=1}=-\sqrt{-1}h'(0)
\left[\frac{c(dx_n)}{|\xi|^4}-4\xi_n\frac{c(\xi')+\xi_nc(dx_n)}{|\xi|^6}\right]-
\frac{2\xi_n\sqrt{-1}\partial_{x_n}c(\xi')(x_0)}{|\xi|^4}.
  \end{align}
By the Cauchy integral formula, then
 \begin{align}
\pi^+_{\xi_n} \partial_{\xi_n}\sigma_{-1}(-2c(w)\nabla_v^{S(TM)}D^{-2})
=&\frac{1}{(\xi_n-i)^{2}}\sum_{j=1}^{n-1}\xi_j  v_{j}c(w)
+\frac{i}{(\xi_n-i)^{2}} v_{n}c(w).
  \end{align}
Combining (4.19) and (4.20), we obtain
  \begin{align}
\label{b26}
\widetilde{\Phi^{**}}_3=&-\frac{1}{2}\int_{|\xi'|=1}\int^{+\infty}_{-\infty} \frac{4i \xi_{n} }{(\xi_{n}-i)(1+\xi_{n}^{2})^{3}}
\sum_{j=1}^{n-1}\xi_{j}v_{j}h'(0){\rm Tr}\big(  c(w) c(\xi') \big)(x_0)d\xi_n\sigma(\xi'){\rm d}x' \nonumber\\
&-\frac{1}{2}\int_{|\xi'|=1}\int^{+\infty}_{-\infty} \frac{-2i\xi_{n} }{(\xi_{n}-i)^{2}(1+\xi_{n}^{2})^{2}}
\sum_{j=1}^{n-1}\xi_{j}v_{j} {\rm Tr}\big(\partial_{x_n}[c(\xi')]c(w)\big) (x_0)d\xi_n\sigma(\xi'){\rm d}x' \nonumber\\
&-\frac{1}{2}\int_{|\xi'|=1}\int^{+\infty}_{-\infty} \frac{1-3\xi_{n}^{2}}{(\xi_{n}-i)^{2}(1+\xi_{n}^{2})^{3}}
h'(0)v_{n}{\rm Tr}\big(c(w)c(dx_{n})\big) (x_0)d\xi_n\sigma(\xi'){\rm d}x' \nonumber\\
=&-\frac{1}{2}\pi^{2}g^{M}(v^{T},w^{T})h'(0){\rm d}x'
-\frac{3}{2}\pi^{2}v_{n}w_{n}h'(0){\rm d}x'.
 \end{align}

{\bf case b)~IV)}~$r=-2,~l=-1,~k=j=|\alpha|=0$

  By (2.14), we get
 \begin{align}
\widetilde{\Phi^{**}}_{4}=-i\int_{|\xi'|=1}\int^{+\infty}_{-\infty}{\rm Tr} [\pi^+_{\xi_n}\sigma_{-2}(-2c(w)\nabla^{ S(TM)}_{v}D^{-2})\times
\partial_{\xi_n}\sigma_{-1}(D^{-1})](x_0){\rm d}\xi_n\sigma(\xi'){\rm d}x'.
  \end{align}
Also, straightforward computations yield
\begin{align}
\partial_{\xi_n}\sigma_{-1}(D^{-1})(x_0)|_{|\xi'|=1}=\sqrt{-1}
\left[\frac{c(dx_n)}{1+\xi_n^2}-\frac{2\xi_nc(\xi')+2\xi_n^2c(dx_n)}{(1+\xi_n^2)^2}\right].
\end{align}
By  the composition formula of pseudodifferential operators, we obtain
\begin{align}
\label{b22}
\sigma_{-2}(-2c(w)\nabla_v^{S(TM)}D^{-2})=&\sigma_{0}(-2c(w)\nabla_v^{S(TM)} )\sigma_{-2}(D^{-2})
+\sigma_{1}(-2c(w)\nabla_v^{S(TM)} )\sigma_{-3}(D^{-2})\nonumber\\
&+\sum_{j=1}^{n}\partial_{\xi_{j}}\sigma_{1}\big(-2c(w)\nabla_v^{S(TM)} \big)
D_{x_{j}}\big(\sigma_{-2}(D^{-2})\big)\nonumber\\
=:&\widetilde{H}_{1}+\widetilde{H}_{2}+\widetilde{H}_{3}.
\end{align}

(1) Let $A(v)=\frac{1}{4}\Sigma_{ij} \langle\nabla_v^L{\tilde{e}_i},\tilde{e}_j\rangle c(\tilde{e}_i)c(\tilde{e}_j)$, then
\begin{align}
\label{b22}
\widetilde{H}_{1}=\sigma_{0}(-2c(w)\nabla_v^{S(TM)} )\sigma_{-2}(D^{-2})=-2c(w)A(v)|\xi|^{-2}.
\end{align}
By the Cauchy integral formula, we obtain
\begin{align}
 \pi^+_{\xi_n}(\widetilde{H}_{1})=& \pi^+_{\xi_n}\Big(-2c(w)A(v)|\xi|^{-2}\Big)
   = \frac{i}{  \xi_n-i }c(w)A(v).
\end{align}
By (3.68), we have
\begin{align}
 {\rm Tr} \Big(c(w)A(v)c(dx_{n})\Big)
=\frac{1}{4}\sum_{i,j,l=1}^{n}\langle\nabla_v^L{\tilde{e}_i},\tilde{e}_j\rangle w_{l} \Big(\delta_{ln}\delta_{ij}
 +2\delta_{ il}\delta_{ jn} +2\delta_{ jl}\delta_{ in}\Big){\rm Tr}_{S(TM)}[{\rm \texttt{id}}]=0.
 \end{align}
Therefore
\begin{align}\label{61}
 -i\int_{|\xi'|=1}\int^{+\infty}_{-\infty}{\rm Tr} [\pi^+_{\xi_n}(\widetilde{H}_{1})\times
\partial_{\xi_n}\sigma_{-1}(D^{-1})](x_0)d\xi_n\sigma(\xi'){\rm d}x'=0.
\end{align}

(2) From Lemma 3.4 and Lemma 3.5, we have
\begin{align}
\widetilde{H}_{2}(x_0)=&\sigma_{1}(-2c(w)\nabla_v^{S(TM)} )\sigma_{-3}(D^{-2})(x_0)\nonumber\\
=&\sum_{j=1}^nv_j(-2c(w))\sqrt{-1}\xi_{j} \times
\Big(-\sqrt{-1}|\xi|^{-4}\xi_k(\Gamma^k-2\delta^k)\nonumber\\
&-\sqrt{-1}|\xi|^{-6}2\xi^j\xi_\alpha\xi_\beta\partial_jg^{\alpha\beta}\Big)(x_0)\nonumber\\
=&\sum_{j=1}^nv_j(-2c(w))\sqrt{-1}\xi_{j} \times
\Big[\frac{-i}{(1+\xi_n^2)^2}\Big(-\frac{1}{2}h'(0)\sum_{k<n}\xi_k
c(\widetilde{e_k})c(\widetilde{e_n})+\frac{5}{2}h'(0)\xi_n\Big)\nonumber\\
&-\frac{2ih'(0)\xi_n}{(1+\xi_n^2)^3}\Big](x_0).
\end{align}
By the Cauchy integral formula, we get
\begin{align}
 \pi^+_{\xi_n}(\widetilde{H}_{2})(x_0)=&\frac{-2-i\xi_{n}}{4(\xi_{n}-i)^{2}}h'(0)
    \sum_{j=1}^{n-1}v_{j}\xi_j\sum_{k<n}\xi_k c(w)c(\widetilde{e_k})c(\widetilde{e_n})\nonumber\\
 &+\frac{-i}{4(\xi_{n}-i)^{2}}h'(0)
     v_{n} \sum_{k<n}\xi_k c(w)c(\widetilde{e_k})c(\widetilde{e_n})\nonumber\\
  &+\frac{1}{2(\xi_{n}-i)^{3}}h'(0)
      \sum_{j=1}^{n-1}v_{j}\xi_jc(w)+\frac{i}{2(\xi_{n}-i)^{3}}h'(0)
       v_{n} c(w).
\end{align}
Combining (4.23) and (4.30), we obtain
\begin{align}\label{61}
&-i\int_{|\xi'|=1}\int^{+\infty}_{-\infty}{\rm Tr} [\pi^+_{\xi_n}(\widetilde{H}_{2})\times
\partial_{\xi_n}\sigma_{-1}(D^{-1})](x_0)d\xi_n\sigma(\xi'){\rm d}x'\nonumber\\
=&-i \int_{|\xi'|=1}\int^{+\infty}_{-\infty} \frac{(-2-i\xi_{n})(i-i\xi_{n}^{2})}{4(\xi_{n}-i)^{4}(\xi_{n}+i)^{2}}
\sum_{j,k=1}^{n-1}\xi_{j}\xi_{k}v_{j}w_{k}h'(0){\rm Tr}_{S(TM)}[{\rm \texttt{id}}](x_0)d\xi_n\sigma(\xi'){\rm d}x' \nonumber\\
&-i\int_{|\xi'|=1}\int^{+\infty}_{-\infty}   \frac{1-\xi_{n}^{2}}{ 2(\xi_{n}-i)^{5}(\xi_{n}+i)^{2}}
 v_{n}w_{n}h'(0){\rm Tr}_{S(TM)}[{\rm \texttt{id}}] (x_0)d\xi_n\sigma(\xi'){\rm d}x' \nonumber\\
&-i \int_{|\xi'|=1}\int^{+\infty}_{-\infty}   \frac{-\xi_{n}}{2(\xi_{n}-i)^{4}(\xi_{n}+i)^{2}}
\sum_{j,k=1}^{n-1}\xi_{j}\xi_{k}v_{n}w_{n}h'(0){\rm Tr}_{S(TM)}[{\rm \texttt{id}}] (x_0)d\xi_n\sigma(\xi'){\rm d}x' \nonumber\\
&-i \int_{|\xi'|=1}\int^{+\infty}_{-\infty}   \frac{i\xi_{n}}{(\xi_{n}-i)^{5}(\xi_{n}+i)^{3}}
\sum_{j,k=1}^{n-1}\xi_{j}\xi_{k}v_{j}w_{k}h'(0){\rm Tr}_{S(TM)}[{\rm \texttt{id}}] (x_0)d\xi_n\sigma(\xi'){\rm d}x'\nonumber\\
=&\frac{1}{6}\pi^{2}g^{M}(v^{T},w^{T})h'(0)dx' -\frac{11}{6}\pi^{2}v_{n}w_{n}h'(0){\rm d}x'.
\end{align}

(3) Also, straightforward computations yield
\begin{align}
\widetilde{H}_{3}(x_0)=&\sum_{j=1}^{n}\partial_{\xi_{j}}\sigma_{1}\big(-2c(w)\nabla_v^{S(TM)} \big)
D_{x_{j}}\big(\sigma_{-2}(D^{-2})\big)(x_0)\nonumber\\
=&-2c(w)\sum_{j=1}^{n}\partial_{\xi_{j}} \big(\sum_{l=1}^{n}v_{l} \sqrt{-1}\xi_{l}\big)
(-\sqrt{-1})\partial_{x_{j}}\big(|\xi|^{-2}\big)(x_0)\nonumber\\
=&2c(w)v_{n}h'(0)|\xi|^{-4}.
\end{align}
A simple computation shows
\begin{align}
 \pi^+_{\xi_n}(\widetilde{H}_{3})(x_0)=2c(w)v_{n}h'(0)\pi^+_{\xi_n}(|\xi|^{-4})(x_0)=\frac{-2-i\xi_{n}}{2(\xi_{n}-i)^{2}}c(w)v_{n}h'(0).
\end{align}
Combining (4.23) and (4.33), we obtain
\begin{align}\label{61}
&-i\int_{|\xi'|=1}\int^{+\infty}_{-\infty}{\rm Tr} [\pi^+_{\xi_n}(\widetilde{H}_{3})\times
\partial_{\xi_n}\sigma_{-1}(D^{-1})](x_0)d\xi_n\sigma(\xi'){\rm d}x'\nonumber\\
=&-i \int_{|\xi'|=1}\int^{+\infty}_{-\infty} \frac{(2+i\xi_{n})(i-i\xi_{n}^{2})}{2(\xi_{n}-i)^{4}(\xi_{n}+i)^{2}}
 v_{n}w_{n}h'(0){\rm Tr}_{S(TM)}[{\rm \texttt{id}}](x_0)d\xi_n\sigma(\xi'){\rm d}x' \nonumber\\
=& -4\pi^{2}v_{n}w_{n}h'(0){\rm d}x'.
\end{align}
Hence, we conclude that,
\begin{align}
\widetilde{\Phi^{**}}_{4}= \frac{1}{6}\pi^{2}g^{M}(v^{T},w^{T})h'(0){\rm d}x' -\frac{35}{6}\pi^{2}v_{n}w_{n}h'(0){\rm d}x'.
\end{align}

{\bf case b)~V)}~$r=-1,~l=-2,~k=j=|\alpha|=0$

  By (2.15), we get
\begin{align}
\widetilde{\Phi^{**}}_{5}=-i\int_{|\xi'|=1}\int^{+\infty}_{-\infty}
{\rm trace} [\pi^+_{\xi_n}\sigma_{-1}(-2c(w)\nabla^{ S(TM)}_{v}D^{-2})\times \partial_{\xi_n}\sigma_{-2}(D^{-1})](x_0)d\xi_n\sigma(\xi'){\rm d}x'.
\end{align}
From Lemma 3.5  and direct computations, we obtain
\begin{align}
\partial_{\xi_n}\sigma_{-2}(D^{-1})(x_0)|_{|\xi'|=1}=&
\frac{1}{(1+\xi_n^2)^3}[(2\xi_n-2\xi_n^3)c(dx_n)p_0c(dx_n)+(1-3\xi_n^2)c(dx_n)p_0c(\xi') \nonumber\\
&+ (1-3\xi_n^2)c(\xi')p_0c(dx_n)-4\xi_nc(\xi')p_0c(\xi')
+(3\xi_n^2-1)\partial_{x_n}c(\xi')+2h'(0)c(\xi')\nonumber\\
&-4\xi_nc(\xi')c(dx_n)\partial_{x_n}c(\xi')+2h'(0)\xi_nc(dx_n)]
+6\xi_nh'(0)\frac{c(\xi)c(dx_n)c(\xi)}{(1+\xi^2_n)^4}.
\end{align}
By the Cauchy integral formula, we obtain
 \begin{align}
\pi^+_{\xi_n}  \sigma_{-1}(-2c(w)\nabla_v^{S(TM)}D^{-2})
=&\frac{-1}{\xi_n-i }\sum_{j=1}^{n-1}\xi_j  v_{j}c(w)
-\frac{i}{\xi_n-i} v_{n}c(w).
  \end{align}
From (4.37) and (4.38) we obtain
  \begin{align}
\label{b26}
\widetilde{\Phi^{**}}_5=&-i\int_{|\xi'|=1}\int^{+\infty}_{-\infty} \frac{(9\xi_{n}^{2}-7)(1+\xi_{n}^{2}) }{2(\xi_{n}-i)(1+\xi_{n}^{2})^{4}}
\sum_{j=1}^{n-1}\xi_{j}v_{j}{\rm Tr}\big(  c(w) c(\xi') \big)(x_0)d\xi_n\sigma(\xi'){\rm d}x' \nonumber\\
&-i\int_{|\xi'|=1}\int^{+\infty}_{-\infty} \frac{(3i\xi_{n}^{3}-13i\xi_{n})(1+\xi_{n}^{2}) }{2(\xi_{n}-i) (1+\xi_{n}^{2})^{4}}
h'(0)v_{n}{\rm Tr}\big(c(w)c(dx_{n})\big) (x_0)d\xi_n\sigma(\xi'){\rm d}x' \nonumber\\
&-i\int_{|\xi'|=1}\int^{+\infty}_{-\infty} \frac{1-3 \xi_{n}^{2} }{(\xi_{n}-i) (1+\xi_{n}^{2})^{3}}
\sum_{j=1}^{n-1}\xi_{j}v_{j} {\rm Tr}\big(\partial_{x_n}[c(\xi')]c(w)\big) (x_0)d\xi_n\sigma(\xi'){\rm d}x' \nonumber\\
&-i\int_{|\xi'|=1}\int^{+\infty}_{-\infty} \frac{4i\xi_{n}  }{(\xi_{n}-i) (1+\xi_{n}^{2})^{3}}
 v_{n} {\rm Tr}\big( c(w)c(\xi')c(dx_{n})\partial_{x_n}[c(\xi')]\big) (x_0)d\xi_n\sigma(\xi'){\rm d}x' \nonumber\\
=& \frac{3}{2}\pi^{2}g^{M}(v^{T},w^{T})h'(0){\rm d}x'
+\frac{9}{2}\pi^{2}v_{n}w_{n}h'(0){\rm d}x'.
 \end{align}

 Now $\widetilde{\Phi^{**}}$ is the sum of the cases  {\bf case b)~I)}-- {\bf case b)~V)}. Combining with the five cases, this yields
\begin{align}\label{795}
\widetilde{\Phi^{**}}=& -\frac{2}{3}\pi^{2} \partial_{x_n}\big(g^{M}(v^{T},w^{T})\big)dx'
  -2\pi^{2}\partial_{x_n}\big(v_{n}w_{n}\big){\rm d}x'\nonumber\\
 &+2\pi^{2}g^{M}(v^{T},w^{T})h'(0)dx' -\frac{4}{3}\pi^{2}v_{n}w_{n}h'(0){\rm d}x'.
\end{align}
For $n=4$, $K(x_0)=-\frac{3}{2}h'(0) $, then
 \begin{align}\label{795}
\widetilde{\Phi^{**}}=& -\frac{2}{3}\pi^{2} \partial_{x_n}g^{M}(v^{T},w^{T}\big){\rm d}x'
  -2\pi^{2}\partial_{x_n}\big(v_{n}w_{n}\big){\rm d}x'\nonumber\\
 &-\frac{4}{3}\pi^{2}g^{M}(v^{T},w^{T}) K(x_0) {\rm d}x'+\frac{8}{9}\pi^{2}v_{n}w_{n}K(x_0){\rm d}x'.
\end{align}
Using an explicit formula for $\Phi^{**}$ and $\widetilde{\Phi^{**}}$, we can
reformulate Proposition 2.5 as follows.
\begin{thm}\label{thmb1}
Let $M$ be a $4$-dimensional oriented
compact spin manifold with boundary $\partial M$ and the metric
$g^{M}$ be defined as above, then we get the following equality:
\begin{align}
\label{b2773}
&{\rm{\widetilde{Wres}}}\Big(\pi^+\big(c(w)(Dc(v)+c(v)D)D^{-1}\big)\circ \pi^+(D^{-2})\Big)\nonumber\\
=&\frac{4\pi^{2}}{3}\int_{M}(Ric(v,w)-\frac{1}{2}s g(v,w)) vol_{g}
+\int_{\partial M}\Big[-\frac{2}{3}\pi^{2} \partial_{x_n}g(v^{T},w^{T}\big)dx'
  -2\pi^{2}\partial_{x_n}\big(v_{n}w_{n}\big)dx'\nonumber\\
   &+2\pi^2 \Big(\sum_{j=1}^{n}g(e_{j},\nabla^{ TM}_{e_{j}}v)g(w,e_{n})-g(w,\nabla^{ TM}_{e_{n}}v)+g(\nabla^{ TM}_{w}v, e_{n}) \Big) \nonumber\\
 &-\frac{4}{3}\pi^{2}g^{M}(v^{T},w^{T}) K(x_0) dx'+\frac{8}{9}\pi^{2}v_{n}w_{n}K(x_0)\Big]{\rm d}x'.
\end{align}
\end{thm}

\section*{ Acknowledgements}
The first author was supported by NSFC. 11501414. The second author was supported by NSFC. 11771070.
 The authors also thank the referee for his (or her) careful reading and helpful comments.

\end{document}